\documentclass[12pt,reqno]{amsart}
\usepackage[headings]{fullpage}
\usepackage{amssymb,amsmath,amscd,bbm,tikz}
\usepackage[all,cmtip]{xy}
\usepackage{url}
\usepackage[bookmarks=true,%
    colorlinks=true,%
    linkcolor=blue,%
    citecolor=blue,%
    filecolor=blue,%
    menucolor=blue,%
    urlcolor=blue,%
    breaklinks=true]{hyperref}
\usepackage{slashed}    
\usepackage{verbatim}
\usepackage[normalem]{ulem}  
\usepackage{diagbox}

\newtheorem{theorem}{Theorem}[section]
\theoremstyle{definition}
\newtheorem{proposition}[theorem]{Proposition}
\newtheorem{lemma}[theorem]{Lemma}
\newtheorem{definition}[theorem]{Definition}
\newtheorem{remark}[theorem]{Remark}

\newtheorem{example}[theorem]{Example}

\def\BZ{\mathbbm Z}
\def\BQ{\mathbbm Q}
\def\BR{\mathbbm R}
\def\BC{\mathbbm C}

\def\BT{\mathbbm T}
\def\BF{\mathbbm F}
\def\calA{\mathcal A}

\def\calT{\mathcal T}

\def\calS{\mathcal S}

\def\calO{\mathcal O}

\def\a{\alpha}
\def\b{\beta}
\def\g{\gamma}

\def\ve{\varepsilon}
\def\th{\theta}

\def\PSL{\mathrm{PSL}}
\def\Jac{\mathrm{Jac}}

\def\Re{\mathrm{Re}}

\def\sgn{\mathrm{sgn}}
\def\be{\begin{equation}}
\def\ee{\end{equation}}

\def\Abar{\overline{A}}
\def\Bbar{\overline{B}}
\def\Cbar{\overline{C}}

\def\sA{\mathsf{A}}
\def\sB{\mathsf{B}}
\def\sC{\mathsf{C}}

\def\vphi{\varphi}
\def\LHS{\mathrm{LHS}}
\def\RHS{\mathrm{RHS}}

\def\lr#1{\langle #1 \rangle}
\def\lb#1#2{\langle #1; #2 \rangle}

\def\CB{\mathrm{BC}}

\def\IF{\mathcal{F}}
\def\IE{\mathcal{E}}
\def\what{\widehat}

\makeatletter

\renewcommand\thepart{\@Roman\c@part}%
\renewcommand\part{%
   \if@noskipsec \leavevmode \fi
   \par
   \addvspace{6.7ex}%
   \@afterindentfalse
   \secdef\@part\@spart}
\def\@part[#1]#2{%
    \ifnum \c@secnumdepth >\m@ne
      \refstepcounter{part}%
      \addcontentsline{toc}{part}{Part~\thepart.\ #1}%
    \else
      \addcontentsline{toc}{part}{#1}%
    \fi
    {\parindent \z@ \raggedright
     \interlinepenalty \@M
     \normalfont
     \ifnum \c@secnumdepth >\m@ne
       \centering\large\scshape \partname~\thepart.%
       \hspace{1ex}%
     \fi%
     \large\scshape #2%
     \markboth{}{}\par}%
    \nobreak
    \vskip 4.7ex
    \@afterheading}
  \def\@spart#1{
  \refstepcounter{part}%
  \addcontentsline{toc}{part}{#1}%
    {\parindent \z@ \raggedright
     \interlinepenalty \@M
     \normalfont
     \centering\large\scshape #1\par}%
     \nobreak
     \vskip 4.7ex
     \@afterheading}
\renewcommand*\l@part[2]{%
  \ifnum \c@tocdepth >-2\relax
    \addpenalty\@secpenalty
    \addvspace{0.75em \@plus\p@}%
    \begingroup
      \parindent \z@ \rightskip \@pnumwidth
      \parfillskip -\@pnumwidth
      {\leavevmode
       \normalsize \bfseries #1\hfil \hb@xt@\@pnumwidth{\hss #2}}\par
       \nobreak
       \if@compatibility
         \global\@nobreaktrue
         \everypar{\global\@nobreakfalse\everypar{}}%
      \fi
    \endgroup
  \fi}

\def\l@subsection{\@tocline{2}{0pt}{2pc}{6pc}{}}
\makeatother

\begin{document}

\title[Quantum dilogarithms over local fields and invariants of 3-manifolds]{
  Quantum dilogarithms over local fields and invariants of 3-manifolds}
\author{Stavros Garoufalidis}
\address{
  International Center for Mathematics, Department of Mathematics \\
  Southern University of Science and Technology \\
  Shenzhen, China \newline
  {\tt \url{http://people.mpim-bonn.mpg.de/stavros}}}
\email{stavros@mpim-bonn.mpg.de}
\author{Rinat Kashaev}
\address{Section de Math\'ematiques, Universit\'e de Gen\`eve \\
2-4 rue du Li\`evre, Case Postale 64, 1211 Gen\`eve 4, Switzerland \newline
         {\tt \url{http://www.unige.ch/math/folks/kashaev}}}
\email{Rinat.Kashaev@unige.ch}

\date{2 June 2023}

\begin{abstract}
  To each local field (including the real or complex numbers)
  we associate a quantum dilogarithm and show that it satisfies a pentagon
  identity and some symmetries. Using an angled version of these quantum
  dilogarithms, we construct three generalized TQFTs in 2+1 dimensions, one given
  by a face state-integral and two given by edge state-integrals. Their partition
  functions rise to distributional invariants of 3-manifolds with torus boundary,
  conjecturally related to point counting of the $A$-polynomial curve. The partition
  function of one of these face generalized TQFTs for the case of the real numbers
  can be expressed either as a multidimensional Barnes-Mellin integral
  or as a period on a curve which is conjecturally the $A$-polynomial curve.
\end{abstract}

\maketitle

{\footnotesize
\tableofcontents
}


\section{Introduction}
\label{sec.intro}

To each local field $F$ (including the real and the complex numbers), we
associate a quantum dilogarithm that satisfies a pentagon identities and
some symmetries. Using an angled version of these quantum dilogarithms,
we construct
\begin{itemize}
\item[(a)]
  a face-type generalized TQFT whose states are elements of
  $\what{F^\times} \times F^\times$ associated to the faces of a triangulation, and
\item[(b)]
  a pair of edge-type generalized TQFTs whose states (in $F^\times$ and
  $\what{F^\times}$, respectively) are associated to the edges of the triangulation.
\end{itemize}
In all three cases, the partition function is a state-integral which is a
distribution on the space of peripheral data. 

These face and edge state-integrals are computable in terms of an ideal triangulation
of a 3-manifold with torus boundary and are conjecturally expressed generically
in terms of an $F$-point counting on the $A$-polynomial curve. This point counting
is reminiscent to motivic integration developed by
Igusa, Kontsevich, Denef, Loeser and others~\cite{Igusa, Denef:rationality,
  Denef:definable, Denef:germs}, as well as to the counting of
$\PSL_2(\BF_p)$-representations and to the Bloch group of a finite
field $\BF_p$ developed by Karuo and Ohtsuki~\cite{Karuo:twist, Karuo:double-twist,
Ohtsuki:bloch}.

Although the above partition functions do not explicitly depend on a Planck's
constant, the size of the residue field of $F$ plays the role of Planck's constant
as is evident from the point-count computations. 

In the case of the real numbers, the $\what{\BR^\times}$ edge state-integrals
can be computed in two different ways: the first is expressed by Mellin--Barnes
integrals of products of the beta function and of the cosine function, and the
second is given by period integrals over a complex curve (the $A$-polynomial curve).
The equality of the two computations follows from a Fourier transform formula and
illustrates why some periods can be expressed in terms of Mellin--Barnes integrals,
analogous to what has been observed in mirror symmetry of Calabi--Yau manifolds by
Passare--Tsikh--Cheshel~\cite{Passare:mellin}. This dual presentation of the partition
function fits well with another instance of face state-integrals, namely the
meromorphic 3D-index of the authors~\cite{GK:mero} whose asymptotics were on the
one hand expressed by beta integrals observed by
Hodgson--Kricker--Siejakowski~\cite{HKS}, and periods observed by Wheeler and
the first author~\cite{GW:asy3D}. Our $\what{\BR^\times}$ edge state-integral
proves the equality of the two, and in particular of~\cite[Eqn.(18)]{GW:asy3D}.
As an example, for the case of the $4_1$, we obtain an identity 
\be
\label{41identity}
\frac{1}{2\pi i}
\int_{\epsilon-i\BR} 
\frac{\operatorname{B}(z,z)^2}{\cos(\pi z)^2} \operatorname{d}\!z
= 2 \int_{-\infty}^1 \frac{dx}{\sqrt{(1-x)(1-x+4x^2)}}
= 5.60241216\dots 
\ee
between a beta-integral and a period of an elliptic curve defined over $\BQ$.

As usual, all three generalized TQFTs give rise to representations of the Ptolemy
groupoid, and in particular, unitary representations of the mapping class groups
of punctured surfaces. A detailed discussion of such representations in the context
of the Teichm\"uller TQFT is given in the thesis of
Piguet~\cite[Sec.3]{Piguet:thesis} and in references therein, following the work
of Andersen and Kashaev. 

We end this introduction by mentioning the prior work on face 
state-integrals~\cite{AK:TQFT} and on edge state-integrals~\cite{KLV, GK:mero}.
In a certain sense, our paper is a continuation of the previous work, with some
interesting new twists, even for the case of the real numbers. 


\part{Local fields}
\label{part1}

\section{Preliminaries}
\label{sec.prelim}

\subsection{Quantum dilogarithms over Gaussian groups}
\label{sub.QDL}

The original motivation for introducing quantum dilogarithms over locally compact
abelian groups (in short, LCA groups) equipped with a Gaussian exponential 
was to define invariants of 3-manifolds using state-integrals which are absolutely
convergent, have universal contours of integration (namely copies of the LCA group in
question) and lead to topological invariants. The latter can be interpreted as
partition functions of complex Chern--Simons theory, or of quantum hyperbolic
geometry. The general theory was introduced originally in the joint work of Andersen
and the second author~\cite{AK:complex} and then gradually extended and refined
in~\cite{Kashaev:YB} and~\cite[App.B]{GK:mero}.

Abstracting from this, we combine an LCA group and a Gaussian exponential
into the notion of a Gaussian group. Let $\BT$ be the multiplicative group of
complex numbers of absolute value $1$.

\begin{definition}
  \label{def:gauss-gr}
  A \emph{ Gaussian group} is an LCA group $\sA$ equipped 
  with a nondegenerate $\BT$-valued quadratic form
  $\langle\cdot\rangle\colon \sA\to\BT$, i.e., a function that satisfies
  $\langle x\rangle=\langle -x\rangle$ for any $x\in\sA$ and its polarization
$$
\langle x;y\rangle:=\frac{\langle x+y\rangle}{\langle x\rangle\langle y\rangle},
\quad \forall (x,y)\in\sA^2,
$$
is a non-degenerate bi-character. The form $\langle \cdot\rangle$ is called
a \emph{ Gaussian exponential} of $\sA$, and the bi-character
$\langle \cdot;\cdot\rangle$ is called its \emph{ Fourier kernel}.
\end{definition}

Note that every Gaussian group is necessarily Pontryagin self-dual, a consequence
of the non-degeneracy of its Fourier kernel. Note that the Gaussian exponential and
its Fourier kernel satisfy $\lr{u}^2= \lb{u}{u}$ for all $u \in \sA$, which implies
that the Fourier kernel determines the square of the Gaussian exponential as is
standard in the theory of quadratic forms.

Note also that a Gaussian group $\sA$ has a canonically normalized Haar measure
determined by the condition of the improper integral 
\be
\label{eq:norm-Haar-m}
\int_{\sA^2}\langle x;y\rangle\operatorname{d}(x,y)=1 \,.
\ee

\begin{definition}
\label{def:qdl-gen} 
A \emph{ quantum dilogarithm} over a Gaussian group $\sA$ is a tempered distribution
over $\sA$ represented by an almost everywhere defined locally integrable function
$\vphi\colon\sA\to \BC$ that satisfies
\begin{enumerate}
\item
  an \emph{ inversion relation}: there exists a non-zero constant
  $c_\vphi\in\BC^\times$ such that
\be
\label{eq:invrel}
\vphi(x)\vphi(-x)=c_\vphi\langle x\rangle
\ee
for almost all $x\in\sA$;
\item
a \emph{ pentagon identity}
\be
\label{eq:5term-qdl}
\vphi(x)\vphi(y)=\gamma_{\sA} \int_{\sA^3}
\frac{\langle x-u;y-w\rangle}{\langle u-v+w\rangle}\vphi(u)\vphi(v)\vphi(w)
\operatorname{d}(u,v,w)
\ee
for almost all pairs $(x,y)\in\sA^2$, where
$\gamma_{\sA}:=\int_{\sA}\langle x\rangle \operatorname{d}\!x$ with the
integrals defined improperly.
\end{enumerate}
\end{definition}

Three examples of quantum dilogarithms for the Gaussian groups
$\BR$, $\BR \times \BR$ and $\BZ \times \BT$ have already appeared in the literature,
and the corresponding invariants are the Teichm\"uller TQFT, the Kashaev--Luo--Vartanov
invariant and the meromorphic 3D-index, respectively. We briefly describe these
examples below.

\begin{example}
\label{ex.AK}  
The field $\BR$ of the real numbers is a Gaussian group with Gaussian exponential
$\langle \cdot \rangle: \BR \to \BT$ given by $\langle x \rangle :=e^{\pi i x^2}$
and Fourier kernel $\langle x;y \rangle:=e^{2 \pi i xy}$
is the usual kernel of the Fourier transform. The Faddeev quantum dilogarithm
$\Phi_{\mathsf{b}}(x)$ of~\cite{Faddeev} for a fixed complex parameter $\mathsf{b}$
with non-zero real part is a quantum dilogarithm which was used in~\cite{AK:TQFT}
to construct a generalized TQFT of face state-integral type. This example is
the origin of Definition~\ref{def:qdl-gen}, and it shows, in particular, that
a given Gaussian group can have many quantum dilogarithms. 
\end{example}

\begin{example}
\label{ex.KLV}
Following~\cite{KLV}, $\BR \times \BR$ is a Gaussian group with
the Gaussian exponential 
\begin{equation}
\langle x\rangle=e^{2\pi i\dot x \ddot x}, \qquad x=(\dot x, \ddot x)
\end{equation}
and the associated Fourier kernel
\begin{equation}
\langle x;y\rangle=e^{2\pi i(\dot x \ddot y+\dot y \ddot x)}. 
\end{equation}
A quantum dilogarithm over this Gaussian group is given by
\begin{equation}
  \phi(x)=\frac{\Phi_{\mathsf{b}}(\dot x+\frac{\ddot x}2)}{
    \Phi_{\mathsf{b}}(\dot x-\frac{\ddot x}2)}, \qquad x=(\dot x, \ddot x),\ 
  y=(\dot y, \ddot y)\,.
\end{equation}
\end{example}

\begin{example}
\label{ex.3Dindex}
Following~\cite[App.A]{GK:mero}, $\BT \times\BZ$ is a Gaussian group 
with the Gaussian exponential given by
\be
\label{eq:gaus-exp}
\langle \cdot\rangle\colon \BT \times \BZ\to \BT,\quad
\langle z,m\rangle:=z^m,\qquad  (z,m)\in \BT \times \BZ,
\ee
and the Fourier kernel
\be
\langle z,m;w,n\rangle:= 
\frac{\langle zw,m+n\rangle}{\langle z,m\rangle\langle w,n\rangle}=z^nw^m \,.
\ee
A quantum dilogarith $\vphi_q$ is given by
\be
\vphi_q \colon \BT \times \BZ\to \BT, \qquad
\vphi_q(z,m) = \frac{(-zq^{1-m};q^2)_\infty}{(-z^{-1}q^{1-m};q^2)_\infty} 
\ee
for $|q| < 1$, where $(x;q)_\infty=\prod_{k=0}^\infty (1-q^k x)$ is the infinite
$q$-Pochhammer symbol.
\end{example}

\begin{remark}
\label{rem.altpentagon}
The pentagon identity~\eqref{eq:5term-qdl} should be interpreted as an integral
identity of tempered distributions. As such, it can equivalently be written
in the form of a distributional integral identity 
\begin{equation}
\label{eq:pent-dist-form}
  \tilde\vphi(x)\tilde\vphi(y)\langle x;y\rangle
  =\int_{\sA}\tilde\vphi(x-z)\tilde\vphi(z)\tilde\vphi(y-z)
  \langle z\rangle\operatorname{d}\!z
\end{equation}
for the (inverse) Fourier transformation
\begin{equation}
\tilde\vphi(x):=\int_{\sA}\frac{\vphi(y)}{\langle x;y\rangle} \operatorname{d}\!y \,.
\end{equation}
Indeed, by using the distributional equality with Dirac's delta-function
\begin{equation}
\label{deltaA}
  \int_{\sA}\langle x;y\rangle\operatorname{d}\!y=\delta_{\sA}(x),
\end{equation}
we have
\begin{small}
\begin{align*}
\tilde\vphi(x)\tilde\vphi(y)\frac{\langle x;y\rangle}{\gamma_{\sA}}
&=
\int_{\sA^2} \frac{\vphi(s)\vphi(t)\operatorname{d}(s,t)}{
\gamma_{\sA}\langle x;s-y\rangle\langle y;t\rangle}
=\int_{\sA^5} \frac{\langle s-u;t-w\rangle\vphi(u)\vphi(v)\vphi(w)}{
\langle x;s-y\rangle\langle y;t\rangle\langle u-v+w\rangle}\operatorname{d}(u,v,w,s,t)
\\
&=\int_{\sA^4} \delta_{\sA}(s-u-y)
\frac{\langle s-u;-w\rangle\vphi(u)\vphi(v)\vphi(w)}{
  \langle x;s-y\rangle\langle u-v+w\rangle}\operatorname{d}(u,v,w,s)\\
&=\int_{\sA^3}\frac{\langle y;-w\rangle\vphi(u)\vphi(v)\vphi(w)}{
  \langle x;u\rangle\langle u-v+w\rangle}\operatorname{d}(u,v,w)
=\int_{\sA^3}\frac{\vphi(u)\vphi(v)\vphi(w)}{
  \langle x;u\rangle\langle y;w\rangle\langle u-v+w\rangle}\operatorname{d}(u,v,w)
\\
&=\int_{\sA^4}\frac{\langle v;z\rangle\vphi(u)\tilde\vphi(z)\vphi(w)}{
  \langle x;u\rangle\langle y;w\rangle\langle u-v+w\rangle}\operatorname{d}(u,v,w,z)
\\
&=\int_{\sA^4}\frac{\langle v'+u+w;z\rangle\vphi(u)\tilde\vphi(z)\vphi(w)}{
  \langle x;u\rangle\langle y;w\rangle\langle v'\rangle}\operatorname{d}(u,v',w,z)
\\
&=\int_{\sA^3}\frac{\vphi(u)\tilde\vphi(z)\vphi(w)\langle z\rangle}{
  \gamma_{\sA}\langle x-z;u\rangle\langle y-z;w\rangle}\operatorname{d}(u,w,z)
=\frac1{\gamma_{\sA}}\int_{\sA}
\tilde\vphi(x-z)\tilde\vphi(z)\tilde\vphi(y-z)\langle z\rangle\operatorname{d}(z) \,.
\end{align*}
\end{small}
\end{remark}

\subsection{Distributions and delta functions over  local fields}
\label{sub.distributions}

In this section we recall some basic facts about test functions and
distributions on local fields $F$, the latter being examples of locally compact
abelian groups. A detailed discussion of these facts may be found for instance
in~\cite{Diestel:haar, Friedlander, Cardenal}. 

Let $\calS(F)$ denote the Schwartz--Bruhat space $\calS(F)$ of Fourier stable
complex-valued test functions. In the case of a non-Archimedean local field, the
test functions are locally constant
compactly supported functions on $F$, taking finitely many
values. Let $\calS'(F)$ denote the dual space of tempered distributions. The Fourier
transform is an automorphism of the space $\calS'(F)$.

We fix a (translationally invariant) Haar measure $\mu_F$ on the additive group
$(F,+)$ of a local field $F$ with the notation for the differential 
$\operatorname{d}_{F}x:=\operatorname{d}\mu_F(x)$.  The multiplicative group
$\sB:=(F^\times,\times)$ is also a locally compact abelian group. We fix its Haar
measure through the following relation for the differentials
\begin{equation}
  \label{2haar}
\operatorname{d}_{\sB}x=\frac{\operatorname{d}_{F}x}{\|x\|},\quad \|x\|:=\|x\|_F^d
\end{equation}
where $d$ is the dimension of the field. For example, we have $d=1$ for
$F=\BR$ or $\BQ_p$ and $d=2$ for $F=\BC$. One can also define $\|x\|$ through the
formula $\mu_F(x A)=\|x\|\mu_F(A)$.

The Dirac delta function $\delta_\sB(x)$ is a tempered distribution defined by the
improper distributional integral
\begin{equation}
  \label{deltaB}
\delta_\sB(x)=\int_{\hat \sB}\alpha(x)\operatorname{d}\!\alpha
\end{equation}
where we choose the normalization  of the Haar measure on $\hat \sB$ by the
condition~\eqref{eq:norm-Haar-m}. In other words, we have
\be
\label{Aab}
\int_{\sA^2}\alpha(x)\beta(y)\operatorname{d}((\alpha,x),(\beta,y))
=1\quad\Leftrightarrow\quad \int_{\sA}\alpha(x)\operatorname{d}(\alpha,x)=1 \,.
\ee
We remark that $\delta_\sB(x)=0$ unless $x=1$. The delta functions $\delta_\sB$
and $\delta_F$ on $\sB$ and $F$ are related by
\begin{equation}
  \label{deltaBF}
\delta_\sB(x) =\delta_F(x-1)\quad \forall x\in \sB \,.
\end{equation}

We begin with a basic lemma on $\delta$-functions.

\begin{lemma}
\label{lem.res0}
{\rm (a)} Let $f\colon \Omega\to F$ be a differentiable function on an open
subset $\Omega\subset F$ of $F$, where the set of zeros $f^{-1}(0)$ is discrete
and each zero $a\in f^{-1}(0)$ is simple in the sense that $f'(a)\ne 0$. Then
\begin{equation}
\label{deltafz}
  \delta_F(f(x))=\sum_{a \,: \, f(a)=0} \frac{1}{\|f'(a)\|}\delta_F(x-a) \,,
\end{equation}
where the derivative is defined in the standard way by using the convergence with
respect to the norm.    
\\
{\rm (b)}
It follows that
\be
\label{fsimple}
\int_\sB \delta_\sB\Big(\frac{1+x}{y}\Big) \operatorname{d}\!x
=
\frac{1}{\| 1-y^{-1} \|} \,.
\ee
\end{lemma}

\begin{proof}
  Part (a) is a special case of a more general formula~\eqref{fdelta} below which
  can be proven on the basis of the change of variables formula for the integral,
  see, for example, \cite{Igusa}. 

  For part (b), we use~\eqref{deltaBF} and~\eqref{2haar} and part (a)
  (with the function $f(x)= \frac{1+x}y-1$ with the only zero $a=y-1$) and compute
  \begin{align*}
\int_\sB \delta_\sB\Big(\frac{1+x}{y}\Big) \operatorname{d}_\sB x
&=\int_F\delta_F\Big(\frac{1+x}{y}-1\Big) \frac{\operatorname{d}_F\!x}{\|x\|} 
=\int_F \|y\| \delta_F(1+x-y) \frac{\operatorname{d}_F\!x}{\|x\|} =
\frac{\|y\|}{\|y-1\|} \,.
  \end{align*}
\end{proof}

Below, we will need a multivariate generalization of Equation~\eqref{deltafz}
which is a local field analogue of Grothendieck's residue theorem, and 
follows from~\cite[Prop.7.4.1]{Igusa}.

\begin{lemma}
\label{lem.res}
Suppose that $f=(f_1,\dots,f_r)$ with $f_i\in F(x)$ with $x=(x_1,\dots,x_r)$ defines
a reduced 0-dimensional scheme $S_f$ with $F^\times$ points
\be
\label{Sf}
S_f(F^\times)=\{a \in (F^\times)^r
\,\, | f_i(a)=0, \,\, i=1,\dots r\}
\ee
which is nondegenerate, i.e., $\Jac(f(a)):=\det(\partial_{z_j} f_i(a)) \neq 0$
for all $a \in S_f$. Then,
\be
\label{fdelta}
\prod_{i=1}^r \delta_F(f_i(x)) = \sum_{a \in S_f(F^\times)}
 \frac{1}{\| \Jac(f(a))\|}\prod_{i=1}^r \delta_F(x_i-a_i) \,.
\ee
\end{lemma}


\section{A quantum dilogarithm over a local field}
\label{sec.pentagon}

\subsection{A Gaussian group associated to a local field}
\label{sub.QDLF}

As in the previous section, let $F$ denote a local field. In characteristic zero,
this means that  $F$ can be $\BR$, $\BC$ or a finite extension of the $p$-adic
numbers $\BQ_p$. Note that points (singletons) in $F$ have measure
zero. The additive group $(F,+)$, the multiplicative group $\sB=F^\times$
and its Pontryagin dual $\hat \sB$ have Haar measures, discussed in detail for
instance in ~\cite{Diestel:haar}. We fix their normalizations as is discussed
in Subsection~\ref{sub.distributions}.

We will denote elements of $\sB$ by $x,y,\dots$ and elements of $\hat \sB$ by
$\alpha,\beta,\dots$. We will denote the canonical pairing
$\hat \sB \times \sB \to \BT$ by $(\alpha,x) \mapsto \alpha(x)$. The LCA groups $\sB$
and $\hat \sB$ can be combined to define 
a self-dual LCA  group 
$\sA = \hat \sB \times \sB$ which is a Gaussian group with Gaussian exponential
\be
\label{gauss}
\langle \cdot \rangle : \sA \to \BT, \qquad \langle (\alpha,x) \rangle = \alpha(x)
\ee
and the associated Fourier kernel 
\be
\label{fourier}
\langle \cdot \, ; \cdot \rangle : \sA^2 \to \BT, \qquad 
\langle (\alpha,x) ; (\beta,y) \rangle = \alpha(y) \beta(x) \,.
\ee
We now define an elementary yet important function on $\sA$ which will play a key
role in this  paper, and which by abuse of language we will call a quantum dilogarithm
over the local field $F$.

\begin{definition}
\label{def.Psi}
We define
\be
\label{QDL}
\vphi: \hat\sB\times(\sB\setminus\{-1\}) \subset \sA \to \BT, \qquad 
\vphi(\a,x)=\a(1+x).
\ee
The value $\vphi(\a,-1)$ is undefined, and, in case of need, one  can assign any
finite value to it \footnote{One natural definition could be
  $\vphi(\a,-1)=e^{i\pi(1-\a(-1))/4}$ which is consistent with the first property
  of the quantum dilogarihm $\vphi(\a,x)\vphi(-\a,1/x)=\a(x)$.} because the subset
$\hat\sB\times\{-1\}\subset\sA$ has measure zero.
\end{definition}
The function $\vphi$ is clearly
a tempered distribution on $\sA$. 


\subsection{The Fourier transform of $\vphi$}
\label{sub.fourier}

In this section we compute the (inverse) Fourier transform $\tilde \vphi$ of $\vphi$
\begin{lemma}
\label{lem.f}
The inverse Fourier transform of $\vphi$ is given by
  \be
  \label{inverseF}
  \tilde \vphi(\b,y):= \int_\sA \frac{ \vphi(\a,x)}{\langle (\a,x);(\b,y) \rangle}
  \operatorname{d}(\a,x)= 
  \frac{1}{\b(y-1)} \frac{1}{\| 1-y^{-1} \|}, \qquad (y \neq 1) \,.
  \ee
\end{lemma}

\begin{proof}
We have:
\be
  \begin{aligned}
\tilde \vphi(\b,y) &= \int_\sA \frac{ \vphi(\a,x)}{\langle (\a,x);(\b,y) \rangle}
\operatorname{d}(\a,x)
= \int_{\sA}
\frac{ \a(1+x)}{\a(y) \b(x)} \operatorname{d}(\a,x) \\
&= \int_{\hat \sB \times \sB}
\a\Big(\frac{1+x}{y} \Big)
\b\Big(\frac{1}{x}\Big)\operatorname{d}(\a,x) 
= \int_\sB \delta_\sB\Big(\frac{1+x}{y}\Big) \langle (-\b,x) \rangle
\operatorname{d}\!x \\
&= \int_\sB \delta_\sB\Big(\frac{1+x}{y}\Big) \langle (-\b,y-1) \rangle
\operatorname{d}\!x = \frac{1}{\b(y-1)}
\int_\sB \delta_\sB\Big(\frac{1+x}{y}\Big) \operatorname{d}\!x \,.
  \end{aligned}
  \ee
  Part (b) of Lemma~\ref{lem.res0} concludes the proof.
\end{proof}
Recall that $\vphi$ is a tempered distribution on $\sA$ and so is $\tilde \vphi$.
The latter is represented by the locally integrable function given in the right hand
side of Equation ~\eqref{inverseF} only when evaluated at a
test function on $\sA$ whose support does not contain $1$.

\subsection{A pentagon identity for $\vphi$}
\label{sub.pentagon}

In this section we give a distributional pentagon identity for $\vphi$.

\begin{theorem}
\label{thm.pentagon}
The function~\eqref{QDL} is a quantum dilogarithm
i.e., it satisfies the inversion relation~\eqref{eq:invrel} and the
pentagon identity~\eqref{eq:pent-dist-form}
\begin{equation}
\label{pentagon}
\tilde \vphi(\a,x)\tilde \vphi(\b,y )
\alpha(y) \beta(x)
=
\int_\sA \tilde \vphi(\b-\g,y/z)
 \tilde \vphi(\g,z)\tilde \vphi(\a-\g,x/z)
\gamma(z)
 \operatorname{d}(\g,z)
\end{equation}
for all $(\a,x), (\b,y) \in \sA$. 
\end{theorem}

\begin{proof}
The inversion relation is elementary. The proof of the pentagon is inspired by
the behavior of the Ptolemy coordinates under a pentagon
transformation~\cite{Goerner:triangulation,GY}. 

Let LHS and RHS denote the left and the right hand sides of~\eqref{pentagon}.
Lemma~\ref{lem.f} implies that
\begin{align}
\label{LHS1}
\LHS &= \frac{f(x)}{\a(x-1)}
\frac{f(y)}{\b(y-1)} \a(y) \b(x)
\\ \label{RHS1}
\RHS &= \int_{\hat \sB \times \sB}
\frac{f(y/z)}{(\b-\g)(y/z-1)}
\frac{f(z)}{\g(z-1)}
\frac{f(x/z)}{(\a-\g)(x/z-1)}
\g(z) \operatorname{d}\!\g \operatorname{d}\!z \,,
\end{align}
where
\be
\label{fdef}
f(x) = \frac{1}{\|1-x^{-1} \|} \,.
\ee
Now, we collect the terms in the left hand side with respect to $\a$ and
$\b$, noting that
\be
\frac{1}{\a(x-1)} \a(y) = \a(\frac{1}{x-1}) \a(y) =
\a(\frac{y}{x-1}) \,.
\ee
Thus, we have
\be
\label{LHS2}
\LHS = f(x) f(y) \a(\frac{y}{x-1}) \b(\frac{x}{y-1}) \,.
\ee
Likewise, collecting terms with respect to $\a$, $\b$ and $\g$ on the RHS,
gives
\be
\begin{aligned}
\label{RHS2}
\RHS &= \int_{\hat \sB \times \sB}
f(y/z) f(z) f(x/z)
\frac{\g(\frac{z}{z-1}(\frac{y}{z}-1)(\frac{x}{z}-1))}{\b(y/z-1) \a(x/z-1)}
\operatorname{d}\!\g \operatorname{d}\!z  \\
&= \int_{\sB}
f(y/z) f(z) f(x/z) \frac{\delta_\sB \left(\frac{z}{z-1}(\frac{y}{z}-1)(\frac{x}{z}-1)
\right)}{\b(y/z-1) \a(x/z-1)} \operatorname{d}\!z  \,.
\end{aligned}
\ee
Now use the fact that $\frac{z}{z-1}(\frac{y}{z}-1)(\frac{x}{z}-1)=1$ is equivalent
to $z(\frac{y}{z}-1)(\frac{x}{z}-1)=z-1$ which is equivalent to 
$z= \frac{xy}{x+y-1}$. Substituting this value in the above equation, the terms
involving $\a$ and $\b$ match those of the LHS and we have
  \be
\label{RHS3}
\frac{\RHS}{\a(\frac{y}{x-1}) \b(\frac{x}{y-1})} = 
f\Big(\frac{x+y-1}{x}\Big)
f\Big(\frac{xy}{x+y-1}\Big)
f\Big(\frac{x+y-1}{y}\Big)
\int_{\sB}
\delta_\sB \Big(\frac{(y-z)(\frac{x}{z}-1)}{z-1}\Big)
  \operatorname{d}\!z  \,.
\ee
Thus, the pentagon identity~\eqref{pentagon} is equivalent to
\be
\label{pentagon2}
f(x) f(y) = f\Big(\frac{x+y-1}{x}\Big)
f\Big(\frac{xy}{x+y-1}\Big)
f\Big(\frac{x+y-1}{y}\Big)
\int_{\sB}
\delta_\sB \Big(\frac{(y-z)(\frac{x}{z}-1)}{z-1}\Big)
  \operatorname{d}\!z  \,.
\ee
Note incidentally that the arguments of $f$ are the ones appearing in the 5-term
relation for the dilogarithm, hence also in the definition of the Bloch
group of $F$; see~\cite{Bloch,Zagier:dilog}.

Using the equation ~\eqref{fdef} for the function $f$, it follows that
Equation~\eqref{pentagon2} is equivalent to
 \be
 \label{pentagon3}
 \int_{\sB}
\delta_\sB \left(\frac{(y-z)(\frac{x}{z}-1)}{z-1}\right)
  \operatorname{d}\!z = \left\|\frac{(x-1)(y-1)}{(x+y-1)^2}\right\| \,.
 \ee
 We apply part (a) of Lemma~\ref{lem.res0} to the function 
\begin{equation}
g(z)=\frac{(y-z)(\frac{x}{z}-1)}{z-1}-1
\end{equation}
with a unique zero:
\begin{equation}
a=\frac{xy}{x+y-1}, \qquad g'(a)=-\frac{(x+y-1)^3}{xy(x-1)(y-1)} 
\end{equation}
and we conclude that
\begin{equation}
 \int_{\sB}
\delta_\sB \left(\frac{(y-z)(\frac{x}{z}-1)}{z-1}\right)
\operatorname{d}\!z =\frac{1}{\|af'(a)\|}=
\left\|\frac{(x-1)(y-1)}{(x+y-1)^2}\right\|. 
\end{equation}
This proves ~\eqref{pentagon3} and concludes the formal proof of the theorem.
 \end{proof}


\section{Angles}
\label{sec.angles}

In this section we introduce an angled version $\Psi_{a,c}$ of the quantum
dilogarithm $\vphi$ which satisfies an angle dependent pentagon
identity~\eqref{eq:angpent} as well as the symmetry relations ~\eqref{eq:symm12}
and~\eqref{eq:symm23}. The function $\Psi_{a,c}$ is the building block for the
partition function of a tetrahedron and the relation it satisfies will be used to
show that the partition function of a triangulation is invariant under 2--3
Pachner moves, and hence a topological invariant. 

Recall that angles were used in previous works (see for
instance~\cite{AK:TQFT,KLV,GK:mero}) as complex deformations of real variables.
In contrast, in our present paper angles (denoted in general by $a$, $b$, $c$)
will be elements of the abelian group
\be
\label{Cgroup}
\sC=\BR\times\sB \,.
\ee
An angle $a=(\dot a, \ddot a)\in \sC$ has a real component $\dot a \in \BR$  
and a local field component $\ddot a \in \sB=F^\times$. 

We define an involution in $\sC$  by the formula 
\be
\bar a:=(\dot a,(\ddot a)^{-1}) \,.
\ee
element $\bar a$ will to be called \emph{conjugate} of $a$.

When three angles $a$, $b$ and $c$ are assigned to a triangle, we will always
assume that they satisfy
\be
\label{abc0}
a+b+c=\varpi:=(1,-1) \,,
\ee
thus $b$ is expressed in terms of $a$ and $c$ by
\be
\label{abc}
b=(1,-1)-a-c=(1-\dot a-\dot c,-(\ddot a\ddot c)^{-1}) \,.
\ee

\begin{definition}
\label{QDLac}  
For $a,c\in \sC$, we define a function
\begin{equation}
\label{QDLabc}
\Psi_{a,c}\colon\hat\sB\times (\sB\setminus\{\ddot a^{-1}\})\to \BC,\qquad
\Psi_{a,c}(\alpha,x)=\frac{1}{\alpha\big((1-\ddot ax)\ddot c\big)}
\frac{\|\ddot ax\|^{\dot c}}{\|1-\ddot ax\|^{1-\dot a}}.
\end{equation}
Notice, that $\Psi_{a,c}$ is defined almost everywhere on $\sA$.
We denote by $\bar \Psi_{a,c}$ the function of
$(\alpha,x)\in\hat\sB\times (\sB\setminus\{\ddot a\})$ defined by
\begin{equation}
  \bar\Psi_{a,c}(\alpha,x)=\overline{\Psi_{\bar a,\bar c}(\alpha,x)}
  =\alpha\big((1-x/\ddot a)/\ddot c\big)
  \frac{\|x/\ddot a\|^{\dot c}}{\|1-x/\ddot a\|^{1-\dot a}}.
\end{equation}
\end{definition}
It follows immediately from the definition that the angle-dependent
function $\Psi_{a,c}$ specializes to $\tilde \vphi$ through the formula
\begin{equation}
\label{specialac}
\bar\Psi_{0,0}(\alpha,x)=\alpha(x)\tilde\vphi(-\alpha,x^{-1}) \,.
\end{equation}

This is a good place to note that the angled version of the quantum dilogarithms
of the three Examples~\ref{ex.AK}, \ref{ex.KLV} and \ref{ex.3Dindex} are functions
in the Schwartz space $\calS(\sA)$ of $\sA$, whereas our function~\eqref{QDLabc} is
only a tempered distribution on $\sA$ under a positivity assumption on angles. This is
the content of the next lemma. 

\begin{lemma}
\label{lem.Psidist}
For all local fields, including the real and the complex numbers, $\Psi_{a,c}$
is a tempered distribution on $\sA$ if
\be
\label{positivity}
\dot a, \, \dot b, \, \dot c \geq 0 \,,
\ee
which is further represented by a locally integrable function when all
the above inequalities are strict.
\end{lemma}

Note that since $\dot a +\dot b + \dot c =1$, the positivity
condition~\eqref{positivity} is equivalent to
\be
\label{positivity2}
1 \geq \dot a, \, \dot b, \, \dot c \geq 0 \,.
\ee

\begin{proof}
When $\dot a=\dot c=0$, the specialization~\eqref{specialac}, together with the
facts that $\tilde\vphi$ is a tempered distribution and $|\a(x)|=1$ for all
$(\a,x) \in \sA$ implies that $\Psi_{a,c}$ is a tempered distribution. When
$\dot a, \dot b, \dot c \geq 0$ with at least one positive, satisfying
$\dot a + \dot b + \dot c =1$, it follows by the symmetry relation~\eqref{eq:symm12}
below, that we can assume that $0 < \dot a, \dot c < 1$. In this case, the
definition of $\Psi_{a,c}$ together with Lemma~\ref{lem.Is} below imply that
$\Psi_{a,c}$ is locally integrable.
\end{proof}

The next lemma was observed by Igusa and others~\cite{Igusa}.
\begin{lemma}
\label{lem.Is}
Fix a non-Archimedean local field $F$ and let $O_F$ denote its ring of integers and
$q$ denote the size of the residue field. Then, 
\be
\label{Ist}
I(s) := \int_{O_F} \|x\|^{s-1} d\mu(x)
\ee
is absolutely convergent if and only if $\Re(s)>0$, in which case it equals to
$\tfrac{1-q^{-1}}{1-q^{-s}}$.
\end{lemma}

\begin{proof}
The proof is elementary using
the fact that $O_F=\{0\} \sqcup_{k=0}^\infty \varpi^k O_F^\times$ where $\varpi$
is a uniformizer of the local field $O_F$, and 
the integral is given by
$$
I(s) = (1-q^{-1}) \sum_{k=0}^\infty q^{-ks}
$$
(where $q$ is the cardinality of the residue field) which is absolutely convergent
if and only if $\Re(s)>0$, in which case it equals to $\tfrac{1-q^{-1}}{1-q^{-s}}$.
\end{proof}

From now on, we will assume that the angles $a,b,c$ with
$a+b+c=(1,-1)$ satisfy the \emph{positivity condition}~\eqref{positivity}.
 The next two theorems give the main properties of the functions $\Psi_{a,c}$
and $\bar\Psi_{a,c}$.

\begin{theorem}
\label{thm.symm}
The functions $\Psi_{a,c}$ and $ \bar\Psi_{a,c}$  satisfy the symmetry relations
\begin{subequations}
  \begin{align}
  \label{eq:symm12}
\Psi_{a,c}(-\alpha,1/x)\alpha(x) &=\bar\Psi_{a,b}(\alpha,x)
\\
  \label{eq:symm23}
  \int_{\hat\sB\times\sB}\Psi_{a,c}(\beta,y)\alpha(x/y)\beta(y/x)
  \operatorname{d}(\beta,y) &=\bar\Psi_{b,c}(\alpha,x)
\end{align}
\end{subequations}
where $b$ satisfies~\eqref{abc}. 
\end{theorem}

\begin{proof}
  We start by the left hand side of \eqref{eq:symm12}
\begin{align*}
  \Psi_{a,c}(-\alpha,1/x)\alpha(x)&= \alpha\big((1-\ddot a/x)\ddot c\big)
  \frac{\|\ddot a/x\|^{\dot c}}{\|1-\ddot a/x\|^{1-\dot a}}\alpha(x)
  = \alpha\big((x-\ddot a)\ddot c\big)
  \frac{\|\ddot a/x\|^{\dot c}}{\|1-\ddot a/x\|^{1-\dot a}}\\
  &= \alpha\big((x-\ddot a)\ddot c\big)
  \frac{\|x/\ddot a\|^{1-\dot c-\dot a}}{\|1-x/\ddot a\|^{1-\dot a}}
  = \alpha\big(-(1-x/\ddot a)\ddot a\ddot c\big)
  \frac{\|x/\ddot a\|^{1-\dot c-\dot a}}{\|1-x/\ddot a\|^{1-\dot a}} \\
  &= \alpha\big((1-x/\ddot a)/\ddot b\big)
  \frac{\|x/\ddot a\|^{\dot b}}{\|1-x/\ddot a\|^{1-\dot a}}=\bar\Psi_{a,b}(\alpha,x).
\end{align*}

Next, we write out the left hand side of \eqref{eq:symm23}, collect the argument
of $\beta$, integrate over $\beta$ and rewrite the remaining integral over the
Haar measure over $F$
\begin{align*}
  \LHS\eqref{eq:symm23}&=\int_{\hat\sB\times\sB}(-\beta)\big((1-\ddot ay)\ddot c\big)
  \frac{\|\ddot a y\|^{\dot c}}{\|1-\ddot ay\|^{1-\dot a}}
 \alpha(x/y)\beta(y/x)\operatorname{d}(\beta,y)\\
 &=\int_{\hat\sB\times\sB}\beta\Big(\frac{y/(\ddot c x)}{1-\ddot ay}\Big)
 \frac{\|\ddot ay\|^{\dot c}}{\|1-\ddot ay\|^{1-\dot a}}
 \alpha(x/y)\operatorname{d}(\beta,y) \\
 &=\int_{F}\delta_F\Big(\frac{y/(\ddot c x)}{1-\ddot ay}-1\Big)
 \frac{\|\ddot ay\|^{\dot c-1}\|\ddot a\|}{\|1-\ddot ay\|^{1-\dot a}}
 \alpha(x/y)\operatorname{d}\!y
\end{align*}
bring the argument of the delta-function to common denominator
\begin{equation*}
  =\int_{F}\delta_F\Big(\frac{(1+\ddot a\ddot cx)y-\ddot cx}{(1-\ddot ay)\ddot cx}\Big)
  \frac{\|\ddot ay\|^{\dot c-1}\|\ddot a\|}{\|1-\ddot ay\|^{1-\dot a}}
 \alpha(x/y)\operatorname{d}\!y
 \end{equation*}
and integrate over $y$ (it is fixed by the value $y=y':=x/(\ddot c^{-1}+\ddot ax)$)
 \begin{align*}
\phantom{1cm} &=\frac{\|\ddot ay'\|^{\dot c}}{\|1-\ddot ay'\|^{-\dot a}}
  \alpha(x/y')  =\frac{\|\ddot ax\|^{\dot c}}{\|\ddot c^{-1}
 +\ddot ax\|^{\dot c}\|1-\frac{\ddot ax}{\ddot c^{-1}+\ddot ax}\|^{-\dot a}}
 \alpha(\ddot c^{-1}+\ddot ax)\\
 &=\frac{\|\ddot ax\|^{\dot c}}{\|\ddot c^{-1}+\ddot ax\|^{\dot c+\dot a}
\|\ddot c\|^{\dot a}}\alpha(\ddot c^{-1}+\ddot ax)
 =\frac{\|\ddot ax\|^{\dot c}\|\ddot c\|^{\dot c}}{\|1+\ddot c\ddot ax\|^{\dot c+\dot a}}
 \alpha\Big(\frac{1+\ddot c\ddot ax}{\ddot c}\Big)\\ 
  &=\frac{\|\ddot c\ddot ax\|^{\dot c}}{\|1+\ddot c\ddot ax\|^{\dot c+\dot a}}
 \alpha\Big(\frac{1+\ddot c\ddot ax}{\ddot c}\Big)
 =\frac{\|x/\ddot b\|^{\dot c}}{\|1-x/\ddot b\|^{1-\dot b}}
 \alpha\Big(\frac{1-x/\ddot b}{\ddot c}\Big)=\bar\Psi_{b,c}(\alpha,x) \,.
 \end{align*}
\end{proof}

\begin{theorem}
\label{thm.pentagonabc}
  Denoting $\Psi_i:=\Psi_{a_i,c_i}$, the following pentagon relation holds
  \begin{equation}
 \label{eq:angpent}
\Psi_1(\alpha,x)\Psi_3(\beta,y)=\int_{\hat\sB\times\sB}
\Psi_0(\alpha-\gamma,x/z)\Psi_2(\gamma,z)\Psi_4(\beta-\gamma,y/z)
\frac{\alpha(y/z)\beta(x/z)}{\gamma(xy/z^2)}\operatorname{d}(\gamma,z)
\end{equation}
provided that
\begin{equation}
  \label{eq:angle-cond}
a_3=a_2+a_4,\quad c_3=a_0+c_4,\quad c_1=c_0+a_4,\quad a_1=a_0+a_2,\quad c_2=c_1+c_3.
\end{equation}
\end{theorem}

Note that Equation~\eqref{eq:angle-cond} implies the balancing condition
$b_0+c_2+b_4=(2,1)$.

\begin{proof}
We write out explicitly the right hand side of~\eqref{eq:angpent}
\begin{align*} 
  \RHS\eqref{eq:angpent}&=\int_{\hat\sB\times\sB}(\gamma-\alpha)
  \big((1-\ddot a_0 x/z)\ddot c_0\big)
  \frac{\|\ddot a_0x/z\|^{\dot c_0}}{\|1-\ddot a_0 x/z\|^{1-\dot a_0}}
 (- \gamma)\big((1-\ddot a_2 z)\ddot c_2\big)
  \frac{\|\ddot a_2z\|^{\dot c_2}}{\|1-\ddot a_2 z\|^{1-\dot a_2}}\\
  &\quad \times(\gamma-\beta)\big((1-\ddot a_4 y/z)\ddot c_4\big)
  \frac{\|\ddot a_4y/z\|^{\dot c_4}}{\|1-\ddot a_4 y/z\|^{1-\dot a_4}}
  \frac{\alpha(y/z)\beta(x/z)}{\gamma(xy/z^2)}\operatorname{d}(\gamma,z)
\end{align*}
collect the arguments of $\alpha,\beta,\gamma$ 
\begin{align*}
&=\int_{\hat\sB\times\sB}\alpha\Big(\frac{y/\ddot c_0}{z- \ddot a_0x}\Big)
  \frac{\|\ddot a_0x\|^{\dot c_0}\|z\|^{1-\dot a_0-\dot c_0}}{\|z-\ddot a_0x \|^{1-\dot a_0}}
  \beta\Big(\frac{x/\ddot c_4}{z- \ddot a_4y}\Big)
  \frac{\|\ddot a_4y\|^{\dot c_4}\|z\|^{1-\dot a_4-\dot c_4}}{\|z-\ddot a_4y \|^{1-\dot a_4}}
  \\
  &\quad \times\gamma\left(\frac{\ddot c_0\ddot c_4(z-\ddot a_0x)(z-\ddot a_4y)}{
\ddot c_2(1-\ddot a_2z)xy}\right)\frac{\|\ddot a_2z\|^{\dot c_2}}{
 \|1-\ddot a_2z\|^{1-\dot a_2}}
\operatorname{d}(\gamma,z)
\end{align*}
integrate over $\gamma$ and rewrite the remaining integral over the Haar
measure over $F$
\begin{align*}
  &=\int_{F}\alpha\Big(\frac{y/\ddot c_0}{z- \ddot a_0x}\Big)
  \frac{\|\ddot a_0x\|^{\dot c_0}\|z\|^{\dot b_0}}{\|z-\ddot a_0x \|^{1-\dot a_0}}
  \beta\Big(\frac{x/\ddot c_4}{z- \ddot a_4y}\Big)\frac{\|\ddot a_4y\|^{\dot c_4}
 \|z\|^{\dot b_4}}{\|z-\ddot a_4y \|^{1-\dot a_4}}\\
  & \quad \times\delta_F\left(\frac{ \ddot c_0\ddot c_4(z-\ddot a_0x)(z-\ddot a_4y)}{
  \ddot c_2(1-\ddot a_2z)xy}-1\right)
  \frac{\|z\|^{\dot c_2-1}\|\ddot a_2\|^{\dot c_2}}{\|1-\ddot a_2z\|^{1-\dot a_2}}
\operatorname{d}\!z
\end{align*}
bring the common denominator in the delta-function (by using the relation
$\ddot b_0\ddot c_2\ddot b_4=1$ which follows from~\eqref{eq:angle-cond}) and
collect the powers of $\|z\|$
\begin{align*}
 &=\int_{F}\alpha\Big(\frac{y/\ddot c_0}{z- \ddot a_0x}\Big)
  \frac{\|\ddot a_0x\|^{\dot c_0}}{\|z-\ddot a_0x \|^{1-\dot a_0}}
  \beta\Big(\frac{x/\ddot c_4}{z- \ddot a_4y}\Big)\frac{\|\ddot a_4y\|^{\dot c_4}
 }{\|z-\ddot a_4y \|^{1-\dot a_4}}\\
  &\quad \times\delta_F\left(\frac{\ddot c_0\ddot c_4z(z-\ddot a_0x -\ddot a_4y
-(\ddot c_0\ddot b_2\ddot c_4)^{-1}xy)}{\ddot c_2(1-\ddot a_2z)xy}\right)
\frac{\|z\|^{\dot b_0+\dot c_2+\dot b_4-1}\|\ddot a_2\|^{\dot c_2}}{
  \|1-\ddot a_2z\|^{1-\dot a_2}} \operatorname{d}\!z
\end{align*}
integrate over $z$ (which is fixed by the value
$z=z':=\ddot a_0x +\ddot a_4y +(\ddot c_0\ddot b_2\ddot c_4)^{-1}xy$)
\begin{equation*}
  =\alpha\Big(\frac{y/\ddot c_0}{z'- \ddot a_0x}\Big)
  \frac{\|\ddot a_0x\|^{\dot c_0}\|x\|}{\|z'-\ddot a_0x \|^{1-\dot a_0}}
  \beta\Big(\frac{x/\ddot c_4}{z'- \ddot a_4y}\Big)\frac{\|\ddot a_4y\|^{\dot c_4}\|y\|
 }{\|z'-\ddot a_4y \|^{1-\dot a_4}}
 \frac{\|z'\|^{\dot b_0+\dot c_2+\dot b_4-2}\|\ddot a_2\|^{\dot c_2}\|\ddot c_2\|}{
   \|1-\ddot a_2z'\|^{-\dot a_2}\|\ddot c_0\ddot c_4\|}.
\end{equation*}
In the obtained expression, along with \eqref{eq:angle-cond}, we use
the equalities
$$
\ddot b_0\ddot c_2\ddot b_4=1,\quad \dot b_0+\dot c_2+\dot b_4=2,
$$
$$
1-\dot a_2z'=(1-\ddot a_1x)(1-\ddot a_3y),
$$
$$
\frac{y/\ddot c_0}{z'- \ddot a_0x}=\frac{1/(\ddot c_0\ddot a_4)}{
  1-x\ddot b_4/(\ddot c_0\ddot b_2)}=\frac{1/\ddot c_1}{ 1-\ddot a_1x},\quad
\frac{x/\ddot c_4}{z'- \ddot a_4y}=\frac{1/(\ddot a_0\ddot c_4)}{
  1-y\ddot b_0/(\ddot b_2\ddot c_4)}=\frac{1/\ddot c_3}{ 1-\ddot a_3y}
$$
and continue the computation as follows:
\begin{align*}
\RHS\eqref{eq:angpent} &=\alpha\Big(\frac{1/\ddot c_1}{ 1-\ddot a_1x}\Big)
\frac{\|\ddot a_0x\|^{\dot c_0}\|x\|\|\ddot c_0/(\ddot c_1 y)\|^{1-\dot a_0}}{
\|1-\ddot a_1x \|^{1-\dot a_0-\dot a_2}} \\
& \qquad \times \beta\Big(\frac{1/\ddot c_3}{ 1-\ddot a_3y}\Big)
\frac{\|\ddot a_4y\|^{\dot c_4}\|y\|\|\ddot c_4/(\ddot c_3 x)\|^{1-\dot a_4}
 }{\|1-\ddot a_3y \|^{1-\dot a_4-\dot a_2}}
\frac{\|\ddot a_2\|^{\dot c_2}\|\ddot c_2\|}{\|\ddot c_0\ddot c_4\|} \\
&=\alpha\Big(\frac{1/\ddot c_1}{ 1-\ddot a_1x}\Big)
\frac{\|x\|^{\dot c_0+\dot a_4}\|\ddot a_0\|^{\dot c_0}
\|\ddot c_0/\ddot c_1\|^{1-\dot a_0}}{\|1-\ddot a_1x \|^{1-\dot a_0-\dot a_2}}\\
 & \qquad
 \times \beta\Big(\frac{1/\ddot c_3}{ 1-\ddot a_3y}\Big)
 \frac{\|y\|^{\dot c_4+\dot a_0}\|\ddot a_4\|^{\dot c_4}
   \|\ddot c_4/\ddot c_3 \|^{1-\dot a_4}
 }{\|1-\ddot a_3y \|^{1-\dot a_4-\dot a_2}}
  \frac{\|\ddot a_2\|^{\dot c_2}\|\ddot c_2\|}{\|\ddot c_0\ddot c_4\|} \\
& =\alpha\Big(\frac{1/\ddot c_1}{ 1-\ddot a_1x}\Big)
  \frac{\|\ddot a_1x\|^{\dot c_1}}{\|1-\ddot a_1x \|^{1-\dot a_1}}
  \beta\Big(\frac{1/\ddot c_3}{ 1-\ddot a_3y}\Big)\frac{\|\ddot a_3y\|^{\dot c_3}
 }{\|1-\ddot a_3y \|^{1-\dot a_3}}  =\Psi_1(\alpha,x)\Psi_3(\beta,y) \,.
\end{align*}
This completes the proof of the theorem.
\end{proof}

\begin{remark}
\label{rem.specialac}
Note that under the specialization of Equation~\eqref{specialac}, the pentagon
identity~\eqref{pentagon} is a special case of the complex conjugate
of~\eqref{eq:angpent} where all angles are set to zero.
\end{remark}

Let $P$ denote the standard ordered pentagon shown on the left hand side of
Figure~\ref{f.23move}.

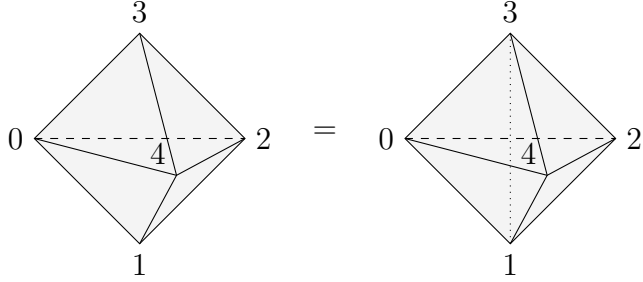
\begin{figure}[!hptb]
\begin{center}
\begin{tikzpicture}[scale=.7,baseline]
\draw[fill=gray!9] (2,0)--(0,2)--(-2,0)--(0,-2)--cycle;
\draw (0,2)--(.7,-.7)--(0,-2);
\draw (-2,0)--(.7,-.7)--(2,0);
\draw[dashed] (-2,0)--(2,0);
\draw (2,0) node[right]{$2$};
\draw (-2,0) node[left]{$0$};
\draw (0,-2) node[below]{$1$};
\draw (0,2) node[above]{$3$};
\draw (.7,-.7) node[above left]{$4$};
\end{tikzpicture}\quad
=\quad
\begin{tikzpicture}[scale=.7,baseline]
\draw[fill=gray!9] (2,0)--(0,2)--(-2,0)--(0,-2)--cycle;
\draw (0,2)--(.7,-.7)--(0,-2);
\draw (-2,0)--(.7,-.7)--(2,0);
\draw[dashed] (-2,0)--(2,0);
\draw[dotted] (0,2)--(0,-2);
\draw (2,0) node[right]{$2$};
\draw (-2,0) node[left]{$0$};
\draw (0,-2) node[below]{$1$};
\draw (0,2) node[above]{$3$};
\draw (.7,-.7) node[above left]{$4$};
\end{tikzpicture}
\end{center}
\caption{The standard ordered 2--3 Pachner move.}
\label{f.23move}
\end{figure}


\begin{proposition}
\label{prop.allpentagons}
The function $\Psi$ satisfies the ordered pentagon identity for all orderings
of the vertices of $P$. 
\end{proposition}

\begin{proof}
  There are $20$ orderings of the vertices of $P$ obtained by applying
  permutations of its five vertices. For each ordering is associated a pentagon
  identity. To $P$ itself is associated the identity~\eqref{eq:angpent}.
  The five equalities of
  ~\cite[Eqn.(4.3)]{Kashaev:4d} implement the permutations $(0,1)$, $(1,2)$, $(2,3)$,
  and $(3,4)$ of the symmetric group of the five vertices $\{0,1,2,3,4\}$ in terms
  of transformations $L$, $M$ and $R$. Then, \cite[Eqn.(4.8)]{Kashaev:4d}
  expresses the transformations $L$, $M$ and $R$ in terms of two symmetries
  $S$ and $T$ of a quantum dilogarithm defined in Eqn.(4.8). i.b.i.d. 
  The symmetries~\eqref{eq:symm12} and~\eqref{eq:symm23} are exactly the $S$
  and $T$ transformations of $\Psi$. 
  It follows that the standard ordered pentagon identity implies all
  other ordered pentagon identities. 
\end{proof}

  

\section{A face-type generalized TQFT}
\label{sec.TQFT1}

In this section we construct a face-type generalized TQFT using a local field $F$. 
In the following constructions, as a combinatorial input, we use the setting
of \emph{ordered $\Delta$-complexes}, see~\cite{MR1867354}. For a $\Delta$-complex
$X$, we let $X_i$ to denote the set of $i$-dimensional cells of  $X$ and $X_{i,j}$
the set of pairs $(a,b)$ where $a\in X_i$ is considered with the ordered
$\Delta$-complex structure of the standard $i$-dimensional simplex $\Delta^i$ and
$b\in (\Delta^i)_j$. As was explained in~\cite{AK:TQFT}, the distributional
properties of the kinematical kernel requires the assumption that the underlying
3-manifold $M$ obtained from $X$ by removing its vertices satisfies the condition
$H_2(M,\BZ)=0$. This ensures that no square of a delta function appears in the
kinematical kernel below.

\subsection{The face-type partition function}
\label{sub.facetype}

Let $X$ be an ordered $\Delta$-complex homeomorphic to an oriented pseudo 3-manifold
with boundary $\partial X$ with its  ordered $\Delta$-complex structure induced
from that of $X$. The kinematical kernel to be defined below depends on only
a Gaussian group and on $X$ but not on a quantum dilogarithm.

Given a self-dual LCA group $\sA$ with a Gaussian exponential
$\langle z\rangle$ and the Fourier kernel $\langle z;w\rangle$, we associate
to $X$ the following \emph{kinematical kernel}
\begin{equation}
  K_X\in \mathcal{S}'(\sA^{(\partial X)_2}\times \sA^{X_3}),\qquad
  K_X(y,z)=\int_{x\in \sA^{X_2}}
  \delta_{\sA^{(\partial X)_2}}(x\vert_{(\partial X)_2}-y)
  \prod_{T\in X_3} K_T(x,z)\operatorname{d}\!x
\end{equation}
where, for a finite set $S$ and a map $f\colon S\to \sA$, we use the  notation
\begin{equation}
\delta_{\sA^S}(f):=\prod_{s\in S} \delta_{\sA}(f(s)),
\end{equation}
and
\begin{equation}
  \label{2delta}
  K_T(x,z):=\langle x_0;z(T)\rangle^{\sgn(T)}
  \delta_{\sA}( x_0- x_1+ x_2)\delta_{\sA}( x_2- x_3+ z(T)),
  \quad x_i:=x(\partial_i T).
\end{equation}
Here, $\calS'(\sA^{(\partial X)_2}\times \sA^{X_3})$ denotes the space of tempered
distributions on the LCA group $\sA^{(\partial X)_2}\times \sA^{X_3}$, the dual of
the space $\calS(\sA^{(\partial X)_2}\times \sA^{X_3})$ of Schwartz-Bruhat
functions on $\sA^{(\partial X)_2}\times \sA^{X_3}$.


Let $\sC$ be another LCA group with a distinguished element $\varpi\in \sC$. A
$\sC$-valued angle structure on $X$ is a map $\theta\colon X_{3,1}\to \sC$ such that,
for any tetrahedron $T\in X_3$, the restriction $\theta(T,\cdot)\colon
(\Delta^3)_1\to \sC$
satisfies the same algebraic conditions as the usual (real valued) dihedral angles
of an ideal hyperbolic tetrahedron  where the value $\pi$ is replaced by $\varpi$. 

Let us  assume now that $\sA=\hat B\times B$ with $B=F^\times$, and let $\sC$ be
defined as in Section~\ref{sec.angles}. The  \emph{dynamical content} of $X$
associated to a $\sC$-valued angle structure $\theta$  and a (symmetric) angle
dependent quantum dilogarithm $\Psi_{a,c}(x)$ over $\sA$, with $a,c\in\sC$,
is defined by
\begin{equation}
  D_{X,\theta}(z)\in \mathcal{S}'(\sA^{X_3}\times \sC^{X_{3,1}}), \quad
  D_{X,\theta}(z)=\prod_{T\in X_3}D_{T,\theta}(z(T))
\end{equation}
where
\begin{equation}
D_{T,\theta}=
\left\{
\begin{array}{cl}
\Psi_{a,c}  &  \text{if }  \sgn(T)=+1, \\
  \bar\Psi_{a,c}  &  \text{if }  \sgn(T)=-1,
\end{array}
\right.
\end{equation}
where $a$ (resp., $b$, $c$) is the angle of the edges $01$ and $23$ (resp.,
$02$ and $13$, and $03$ and $12$) as in Figure~\ref{f.tetangles}.

\begin{figure}[!htpb]
\begin{center}
\begin{tikzpicture}[scale=1.3,baseline=-3]
\draw (0,0) circle  (1);
\draw (90:1)--(0,0)--(-30:1) (0,0)--(-150:1);
\draw (90:.5) node[fill=white]{\small $a$};
\draw (-90:1) node[fill=white]{\small $a$};
\draw (-30:.5) node[fill=white]{\small $b$};
\draw (150:1) node[fill=white]{\small $b$};
\draw (-150:.5) node[fill=white]{\small $c$};
\draw (30:1) node[fill=white]{\small $c$};
\draw (90:0.1) node[right]{\small 0};
\draw (90:1.1) node[above]{\small 1};
\draw (-30:1.1) node[right]{\small 2};
\draw (-150:1.1) node[left]{\small 3};
\end{tikzpicture}
\caption{The angles of an ideal tetrahedron with ordered vertices.}
\label{f.tetangles}
\end{center}
\end{figure}
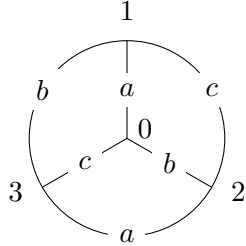

The associated \emph{partition function} of the pair $(X,\theta)$ is the
the push-forward along the tetrahedral variables of the product of the
kinematical kernel with the dynamical content.
\begin{equation}
  \label{Idef}
\IF_{F,X,\theta}= \int_{z\in \sA^{X_3}} K_X D_{X,\theta} \operatorname{d}\!z \,.
\end{equation}

Unlike the case of the three examples \ref{ex.AK}, \ref{ex.KLV} and \ref{ex.3Dindex}
where the kinematical kernel was a distribution and the dynamical content was a
function in the Bruhat--Schwartz space and the two were contracted, here both the
kinematical kernel and the dynamical content are distributions and we multiply them
together, 
and then we push them forward.
  
This partition function has three important properties:
\begin{itemize}
\item[(P1)] The partition function $\IF_{F,X,\theta}$ is invariant under ordered,
  angled, 2--3 Pachner moves. 
\item[(P2)] As a function of $\th=(\dot \th, \ddot \th)$, it is distributional
  on $\ddot \th$, and for every test function $\psi \in \calS(\calA_X)$ (where
  $\calA_X$ is the affine space where $\ddot \th$ takes values), the function
  $\dot \th \mapsto IF_{F,X,\theta}$ extends to a meromorphic function of $q^{\dot \th}$.
\item[(P3)] The partition functon $\IF_{F,X,\theta}$ is angle gauge-invariant
  (see below Subsection~\ref{sub.gauge}).  
\end{itemize}

These properties imply the following theorem as was explained exactly 
in~\cite{AK:TQFT} as well as in propositions 3.2, 3.3 and 3.4 of~\cite{GK:mero}.

\begin{theorem}
\label{thm.Iinvariance}
The distribution $\IF_F$ descends to an invariant $\IF_{F,M}(\lambda,\mu)$
of a compact, oriented 3-manifold $M$ satisfying $H_2(M,\BZ)=0$, where 
$(\lambda,\mu) \in H_1(\partial M,\sC)$.
\end{theorem}
The invariant is distributional on $(\ddot \lambda, \ddot \mu) \in H_1(\partial M,
F^\times)$ and evaluated at a test function, extends to a meromorphic function of
$(q^{\dot \lambda}, q^{\dot \mu})$ where
$(\dot \lambda,\dot \mu) \in H_1(\partial M, \BR)$.

We call the above invariant a face state-integral, following the fact that
the states are assigned to the faces of the triangulation.

In the remaining of the section we discuss the three properties of the
partition function. 
The invariance of the partition function under all ordered angled Pachner moves
follows from Proposition~\ref{prop.allpentagons}. 

\subsection{Integration and point counting}
\label{sub.integration}

In this section we discuss property P2.
The partition function~\eqref{Idef} involves integration on $\sA^{2N}$
(where $\sA=\hat\sB \times \sB$ and $\sB=F^\times$) of products of delta functions
times a product of $\vphi$-functions times a product of positive powers of $\|x_i\|$
and $\|1-x_i\|$ times a fixed test function $\psi$. (Recall that test functions
on non-Archimedean local fields $F$ are compactly supported and take finitely
many values). These integrals reduce to
integrals on $\sB^N$ of products of delta
functions of rational functions $f(x) \in F(x)$ of $N$-variables $x=(x_1,\dots,x_N)$,
times a product of positive powers of $\|x_i\|$ and $\|1-x_i\|$ times a test function.
Such a functional integration was developed, among others, by Denef and Loeser,
who studied the Igusa local zeta functions from the point of view of counting
solutions to polynomial equations modulo $\BZ/p^n\BZ$. A motivic version of that
integration was introduced by Kontsevich in 1995, an arithmetic version of which
was given in~\cite{Denef:definable} and a geometric one
in~\cite{Denef:rationality, Denef:germs}. 

Although the integration is defined analytically, it is expressed in terms of
point counting solutions of equations modulo $\BZ/p^n\BZ$ for fixed $p$ and
varying $n$. This, together with Hironaka's resolution of singularities,
inclusion-exclusion, and a local calculation identifies the integrals over $\sA^{2N}$
in terms of rational functions. In particular, Denef~\cite[Thm.3.2]{Denef:rationality}
proves that if $S$ is a boolean combination of subsets of $F^m$, with $S$ compact,
and $g \in F[x]$, $x=(x_1,\dots,x_m)$, then
\be
\label{Zgs}
I(s)=\int_S \|g(x)\|^s \operatorname{d} x
\ee
is a rational function of $p^{-s}$. In~\cite{Denef:definable} a stronger result
was proven: the rational function of $p^{-s}$ is independent of $p$ (the characteristic
of the residue field of the local field $F$) for all but finitely many $p$.
Moreover, in~\cite{Loeser:many} Loeser gives a multivariable generalization
of the above results for distributional integrals of the form
\be
\label{Zggs}
I(s)=\int_S \prod_{j=1}^r \|g_j(x)\|^{s_j} \operatorname{d} x
\ee
for $s=(s_1,\dots,s_k)$, defined initially for $\Re(s_j)>0$, and analytically continued
as meromorphic functions with poles on a finite union of linear hyperplanes.
This implies property (P2) of the partition function~\eqref{Idef}. 

  
\subsection{Angle gauge transformations}
\label{sub.gauge}

In this section we discuss the invariance of the partition function~\eqref{Idef}
under angled gauge transformation. 

Following~\cite{AK:TQFT}, we let $SP_n$ denote a bipyramid with two vertices
labeled $0$ and $1$ and with basis a polygon $P_n$ with $n$ sides, see
Figure~\ref{f.SPn}. Using the edge $e:=01$ of $SP_n$, we can triangulate $SP_n$
into $n$ positively oriented tetrahedra with common edge $e$.

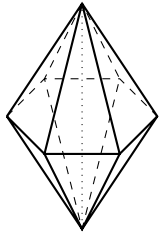
\begin{figure}[!htpb]
\begin{center}
\begin{tikzpicture}
\coordinate (n) at (1,3);
\coordinate (s) at (1,0);
\coordinate (w) at (0,1.5);
\coordinate (e) at (2,1.5);
\coordinate (nw) at (.5,2);
\coordinate (sw) at (.5,1);
\coordinate (ne) at (1.5,2);
\coordinate (se) at (1.5,1);
\draw [dashed] (w)--(nw)--(ne)--(e)(n)--(nw)--(s)--(ne)--(n);
\draw[dotted] (n)--(s);
\draw[thick] (w)--(sw)--(se)--(e)(w)--(s)--(e)--(n)--(w)(n)--(sw)--(s)--(se)--(n);
\end{tikzpicture}
\caption{The bipyramid $SP_n$ with $n=6$.}
\label{f.SPn}
\end{center}
\end{figure}

Enumerating the tetrahedra by $i=1,2,\dots,n$ in the cyclic order around
the common edge $e$, the partition function of $SP_n$ is given by
\begin{equation}
  \label{eq:suspensionSPn}
  \IF_{F,SP_n,\theta}= \int_{\sA^n}\prod_{i=1}^n
  \langle x_i,u_i|T(a_i,c_i)|y_i,u_{i+1}\rangle\operatorname{d}(u_1,\dots,u_n)
\end{equation}
where we identify $u_{n+1}=u_1$ and (in this section only) we use the notation
\begin{equation}
\langle x,u|T(a,c)|y,v\rangle:=\langle x;v-u\rangle \delta_F(x-y+u)\Psi_{a,c}(v-u).
\end{equation}

The variables $x_i$ and $y_i$ in~\eqref{eq:suspensionSPn} are the state variables
on the  boundary of $SP_n$ respectively associated to the minimal and the next to
the minimal faces of the $i$-th tetrahedron, while $a_i$ and $c_i$ are the angle
variables of the $i$-th tetrahedron associated to the edges 
connecting the minimal vertex to the next minimal and the maximal vertices
respectively. 

The \emph{gauge transformation} associated with edge $e$ by amount $\lambda\in \sC$
is the transformation of the angle structure $\theta\mapsto \theta'$ given by the
simultaneous shift  of all angles $c_i$ by the value $\lambda$:
\begin{equation}
c_i\mapsto c_i'=c_i+\lambda,\quad \forall i\in\{1,2,\dots,n\}.
\end{equation}
We claim that
\begin{equation}
\IF_{F,SP_n,\theta'}=\IF_{F,SP_n,\theta}\|\ddot w_e\|^{\dot \lambda}
\end{equation}
where 
\begin{equation}
w_e=\sum_{i=1}^na_i
\end{equation}
is the total angle around edge $e$ which is given by contributions
from all tetrahedra to the edge $e$. Indeed, we have 
\begin{equation}
  \frac{\Psi_{a,c+\lambda}(\a,x)}{\Psi_{a,c}(\a,x)}
  =\frac{\|\ddot a x\|^{\dot\lambda}}{\a(\ddot\lambda)}
  =\|\ddot a\|^{\dot\lambda}f_\lambda(\a,x)
\end{equation}
where $f_\lambda\colon \sA\to \mathbb{C}^\times$ is a complex (multiplicative)
character defined by
\begin{equation}
\label{flambda}
f_\lambda(\a,x)=\frac{\|x\|^{\dot\lambda}}{\alpha(\ddot\lambda)} \,.
\end{equation}
Thus, we have
\begin{multline}
  \IF_{F,SP_n,\theta'}= \int_{\sA^n}\prod_{i=1}^n
  \langle x_i,u_i|T(a_i,c_i+\lambda)|y_i,u_{i+1}\rangle\operatorname{d}(u_1,\dots,u_n)
  \\
  = \int_{\sA^n}\prod_{i=1}^n\langle x_i,u_i|T(a_i,c_i)|y_i,u_{i+1}\rangle
  \|\ddot a_i\|^{\dot\lambda}f_\lambda(u_{i+1}-u_i)\operatorname{d}(u_1,\dots,u_n)\\
  =\Big\|\prod_{i=1}^n\ddot a_i\Big\|^{\dot\lambda} \int_{\sA^n}
  \prod_{i=1}^n\langle x_i,u_i|T(a_i,c_i)|y_i,u_{i+1}\rangle
  \frac{f_\lambda(u_{i+1})}{f_\lambda(u_i)}\operatorname{d}(u_1,\dots,u_n)
  =\|\ddot w_e\|^{\dot\lambda} \IF_{F,SP_n,\theta}.
\end{multline}
This result implies the following gauge invariance property of the partition function.

\begin{lemma}
The partition function  $\IF_{F}$  is invariant under any gauge transformation
associated with an internal edge, provided the $\sB$-component of the total angle
around that edge is of unit norm (in particular, this is the case for a balanced
edge corresponding to the total angle $(2,1)$). 
\end{lemma}



\section{Computations}
\label{sec.computations}

\subsection{Preliminaries}
\label{sub.prelim}

To compute the invariant defined in the previous section, we fix a triangulation
$X$ with $N$ ordered tetrahedra and $2N$ faces. We assign variables
$z_0,\dots,z_{N-1} \in \sA$ to each tetrahedron and $x_0, \dots, x_{2N-1} \in \sA$
to each face. Then, we integrate first over the $x$-variables and then over the
$z$-variables. There are two $\delta$-function linear equations per each tetrahedron,
giving rise to a total of $2N$ linear equations for $x$. Typically, we can solve
those uniquely, and hence express $x$ as a $\BQ$-linear combination of $z$. Doing
so, we express the invariant as an integral over $\sA^N$.

  

\subsection{The trefoil knot}
\label{sub.31}


Let $X$ be an ideal triangulation of the complement of the trefoil knot
in $S^3$ with two positive tetrahedra $T_0$ and $T_1$
with the edge and face identifications given by
\begin{equation}
\label{31faces}
\begin{array}{c|cccccc|}
\hbox{\diagbox[width=2.5em]{\tiny{tet}}{\tiny{edge}}} &  01 & 02 & 03 & 12 & 13 & 23 \\
\hline
    0  &            0 &  0 &  1 &  0 &  0 &  0 \\
    1  &            0 &  0 &  1 &  0 &  0 &  0 \\
  \end{array}
  \qquad\qquad
\begin{array}{c|cccc|} 
\hbox{\diagbox[width=2.5em]{\tiny{tet}}{\tiny{face}}} & 012 & 013 & 023 & 123 \\ \hline
    0  &             0 &  1 &  2 &  3  \\
    1  &             3 &  2 &  1 &  0  \\
 \end{array}
  \end{equation}
Labeling each face by a variable $x_j$ for $j=0,\dots,3$, 
the delta function equations of the kinematical kernel are
$$
-x_0 + x_1 + z_0=0, \qquad x_1 - x_2 + x_3 =0, \qquad x_2 - x_3 + z_1 =0,
\qquad x_0 - x_1 + x_2=0
$$
with unique solution
$$
(x_0,x_1,x_2,x_3)=(z_0 - z_1, -z_1, -z_0, -z_0 + z_1) \,.
$$
The kinematical kernel is
\begin{equation}
\label{31kin}
  K_X(z)= \frac{\lb {z_0}{z_1}^2}{\lb {z_0}{z_0} \lb {z_1}{z_1}} = 
\prod_{0 \leq i, j \leq 1} \lb {z_i}{z_j}^{\frac{1}{2}Q_{ij}}, \qquad
Q=\begin{pmatrix}
  -2 & 2 \\
  2 & -2 
\end{pmatrix} \,.
\end{equation}
Indeed, $K_X(z) = \langle x_3; z_0 \rangle \langle x_0; z_1 \rangle =
\langle  -z_0 +z_1; z_0 \rangle \langle z_0-z_1; z_1 \rangle =
\lb {z_0}{z_1}^2 \lb {z_0}{z_0}^{-1} \lb {z_1}{z_1}^{-1}$.
Writing
$$
z_0=(\alpha,x), \qquad z_1=(\beta,y)
$$
we obtain that
$$
K_X(z)=\alpha(x^{-2}y^{2}) \, \beta(x^{2} y^{-2} ) \,.
$$
The balancing condition on the edges is
\be
\label{31balance}
c_0 + c_1 =(2,1)
\ee
from which we can express all angles in terms of $a_0$, $a_1$ and $c_1$.
After balancing the angles at all edges, the angle holonomy of half of the
longitude is given by
\be
\label{31lon}
\lambda = 2 a_0 -2 a_1 -c_0 \,.
\ee
The invariant is given by
\begin{equation}
  \IF_{F,X,\theta}=\int_{\hat\sB^2\times\sB^2}(\alpha(y/x)\beta(x/y))^{2}
  \Psi_{a_0,c_0}(\alpha,x)\Psi_{a_1,c_1}(\beta,y)\operatorname{d}(\alpha,\beta,x,y)
\end{equation}
After changing variables to $x \mapsto x/\ddot a_0$, $y \mapsto y/\ddot a_1$
and integrating out the $\alpha$ and $\beta$ variables using~\eqref{deltaB}
and~\eqref{deltaBF}, the integral becomes
\begin{align*}
  \IF_{F,X,\theta} =&\int_{\hat\sB^2\times\sB^2}
  \alpha\Big(\frac{y^2/\ddot c_0}{(1-\ddot a_0x)x^2}\Big)
  \frac{\|\ddot a_0x\|^{\dot c_0}}{\|1-\ddot a_0x\|^{1-\dot a_0}}
  \beta\Big(\frac{x^2/\ddot c_1}{(1-\ddot a_1y)y^2}\Big)
  \frac{\|\ddot a_1y\|^{\dot c_1}}{\|1-\ddot a_1y\|^{1-\dot a_1}}
  \operatorname{d}(\alpha,\beta,x,y)\\
=&\int_{\hat\sB^2\times\sB^2}
  \alpha\Big(\frac{y^2(\ddot a_0/\ddot a_1)^2/\ddot c_0}{(1-x)x^2}\Big)
  \frac{\|x\|^{\dot c_0}}{\|1-x\|^{1-\dot a_0}}
  \beta\Big(\frac{x^2(\ddot a_1/\ddot a_0)^2/\ddot c_1}{(1-y)y^2}\Big)
  \frac{\|y\|^{\dot c_1}}{\|1-y\|^{1-\dot a_1}}
  \operatorname{d}(\alpha,\beta,x,y)\\
  \stackrel{\eqref{deltaB},\eqref{deltaBF}}{=} &
  \int_{\sB^2} \delta_F\Big(\frac{y^2(\ddot a_0/\ddot a_1)^2/\ddot c_0}{(1-x)x^2}-1\Big)
\frac{\|x\|^{\dot c_0}}{\|1-x\|^{1-\dot a_0}}
\delta_F\Big(\frac{x^2(\ddot a_1/\ddot a_0)^2/\ddot c_1}{(1-y)y^2}-1\Big)
\frac{\|y\|^{\dot c_1}}{\|1-y\|^{1-\dot a_1}}\operatorname{d}(x,y) \\  
\stackrel{\eqref{2haar}}=&\int_{F^2}
\delta_F\Big(\frac{y^2(\ddot a_0/\ddot a_1)^2/\ddot c_0}{(1-x)x^2}-1\Big)
\frac{\|x\|^{\dot c_0-1}}{\|1-x\|^{1-\dot a_0}}
\delta_F\Big(\frac{x^2(\ddot a_1/\ddot a_0)^2/\ddot c_1}{(1-y)y^2}-1\Big)
\frac{\|y\|^{\dot c_1-1}}{\|1-y\|^{1-\dot a_1}}\operatorname{d}(x,y) \,.
\end{align*}
Using the edge-balancing equations we find that the delta function equations
take the form
\be
\label{31ge}
1-x = \ve x^{-2} y^2, \qquad 1-y = \ve^{-1} \, x^{2} y^{-2}
\ee
for 
$\ve:=\ddot \lambda$.
This, together with the edge-balancing conditions implies that
\be
\label{31xnorms}
\frac{\|x\|^{\dot c_0-1}}{\|1-x\|^{1-\dot a_0}}
  \frac{\|y\|^{\dot c_1-1}}{\|1-y\|^{1-\dot a_1}}
  = \frac{1}{\| \ve \|^{\dot \theta_\mu}} \|x\|^{e_0} \|y\|^{e_1} 
\ee
where
\be
\label{d31}
\begin{aligned}
(e_0,e_1) &:= (\dot c_0-1, \dot c_1-1)
- (1-\dot a_0, 1-\dot a_1) Q  \\
&= (\dot c_0, \dot c_1)+(\dot a_0, \dot a_1) Q
-(1,1)-Q.(1,1) \\ &= \dot \lambda(-1,1) +(-1,-1) 
\end{aligned}
\ee
and
\be
\label{dmer31}
\dot \mu=(1,-1).(1-\dot a_0, 1-\dot a_1) = -\dot a_0 + \dot a_1 \,.
\ee
The above discussion, combined with Lemma~\eqref{lem.res}, implies that
$I_F(3_1,\theta)$ depends only on $\lambda=
(\dot \lambda, \ddot \lambda)$ and $\dot \theta_\mu$,
and is given by
\be
\label{IF31a}
\IF_{F,3_1}(\ve,s,t) = \frac{1}{\| \ve \|^t}
\sum_{(x,y) \in X_\ve(F^\times)} \frac{ \|x\|^{-s-1} \|y\|^{s-1}
}{\|\Jac(f(x,y)) \|}
\ee
where $\ve=\ddot \lambda \in F^\times$, $s=\dot \lambda \in \BR$ and
$t=\dot \mu \in \BR$.

\begin{remark}
\label{rem.chiral}
Note that the invariant of $3_1$ is chiral. Indeed, the mirror image of
$3_1$, obtained by changing the orientation, corresponds to transformation 
$$
\ddot\lambda\mapsto1/\ddot\lambda,\quad \dot\lambda\mapsto\dot\lambda,
\quad\dot\mu\mapsto-\dot\mu \,.
$$

Thus, our invariant is sensitive to the orientation of $X$.
\end{remark}

\subsection{The figure-eight knot}
\label{sub.41}


Let $X$ be an ideal triangulation of the complement of the figure-eight knot
in $S^3$ with one positive tetrahedron $T_0$ and one negative tetrahedron $T_1$
with the edge and face identifications given by
\begin{equation}
\label{41faces}
\begin{array}{c|cccccc|}
\hbox{\diagbox[width=2.5em]{\tiny{tet}}{\tiny{edge}}} &  01 & 02 & 03 & 12 & 13 & 23 \\
\hline
    0  &            0 &  1 &  0 &  1 &  1 &  0 \\
    1  &            1 &  0 &  1 &  0 &  0 &  1 \\
  \end{array}
  \qquad\qquad
\begin{array}{c|cccc|} 
\hbox{\diagbox[width=2.5em]{\tiny{tet}}{\tiny{face}}} & 012 & 013 & 023 & 123 \\
\hline
    0  &             0 &  1 &  2 &  3  \\
    1  &             2 &  3 &  0 &  1  \\
\end{array}
\end{equation}
Labeling each face by a variable $x_j$ for $j=0,\dots,3$, 
the delta function equations of the kinematical kernel are
\begin{align*}
  -x_0 + x_1 + z_0&=0,
  & -x_2 + x_3 + z_1&=0, \\
  x_1 - x_2 + x_3&=0,
& -x_0 + x_1 + x_3&=0,  
\end{align*}
with unique solution
$$
(x_0,x_1,x_2,x_4)= (z_0 + z_1, z_1, z_0 + z_1, z_0) \,.
$$
The kinematical kernel takes the form
\begin{equation}
\label{41kin}
K_X(z)=\frac{\lb {z_0}{z_0}}{\lb {z_1}{z_1}} =
\prod_{0 \leq i, j \leq 1} \lb {z_i}{z_j}^{\frac{1}{2}Q_{ij}}, \qquad
Q=\begin{pmatrix}
  2 & 0  \\
  0 & -2 
  \end{pmatrix} \,.
\end{equation}
Writing
$$
z_0=(\alpha,x), \qquad z_1=(\beta,y)
$$
we obtain that
$$
K_X(z)=\alpha(x^{2}) \, \beta(y^{-2} ) \,.
$$
The balancing condition on the edges is
\be
\label{41balance}
2 a_0 + c_0 + 2 b_1 + c_1 =(2,1) 
\ee
from which we can express all angles in terms of $a_0$, $a_1$ and $c_1$.
After balancing the angles at all edges, the angle holonomy of half of the
longitude is given by
\be
\label{41lon}
\lambda = 2 a_0 + c_0 - \varpi = 2 a_0 + c_0 - \varpi \,.
\ee

The integral is given by
\be
\label{Z41b}
\begin{aligned}
\IF_{F,X,\theta}= & \int_{\hat\sB^2\times\sB^2}
  \alpha\Big(\frac{x^2}{(1-\ddot a_0x)\ddot c_0}\Big)
  \beta\Big(\frac{y^{-2}(1-y \ddot a_1^{-1})}{\ddot c_1}\Big) 
  \frac{\|\ddot a_0 x\|^{\dot c_0}}{\|1-\ddot a_0x\|^{1-\dot a_0}}
  \frac{\|\ddot a_1 y\|^{\dot c_1}}{\|1-\ddot a_1y\|^{1-\dot a_1}}
  \operatorname{d}(\alpha,\beta,x,y) \\
  =&
  \int_{\hat\sB^2\times\sB^2}
  \alpha\Big(\frac{x^2}{(1-x)\ddot c_0 \ddot a_0^2}\Big)
  \beta\Big(\frac{1-y}{y^2 \ddot c_1 \ddot a_1^2}\Big)
  \frac{\|x\|^{\dot c_0}}{\|1-x\|^{1-\dot a_0}}
  \frac{\|y\|^{\dot c_1}}{\|1-y\|^{1-\dot a_1}}
  \operatorname{d}(\alpha,\beta,x,y) \\
  \stackrel{\eqref{deltaB},\eqref{deltaBF}}{=}   & \int_{\sB^2}
  \delta_F\Big(\frac{x^2}{(1-x)\ddot c_0 \ddot a_0^2} -1\Big)
  \delta_F\Big(\frac{1-y}{y^2 \ddot c_1 \ddot a_1^2} -1\Big)
  \frac{\|x\|^{\dot c_0}}{\|1-x\|^{1-\dot a_0}}
  \frac{\|y\|^{\dot c_1}}{\|1-y\|^{1-\dot a_1}}
  \operatorname{d}(x,y) \\
  \stackrel{\eqref{2haar}}=&\int_{F^2}
  \delta_F\Big(\frac{x^2}{(1-x)\ddot c_0 \ddot a_0^2} -1\Big)
  \delta_F\Big(\frac{1-y}{y^2 \ddot c_1 \ddot a_1^2} -1\Big)
  \frac{\|x\|^{\dot c_0-1}}{\|1-x\|^{1-\dot a_0}}
  \frac{\|y\|^{\dot c_1-1}}{\|1-y\|^{1-\dot a_1}}
  \operatorname{d}(x,y) \\
  =&\int_{F^2}
  \delta_F\Big(-\frac{x^2}{(1-x)\ddot \lambda} -1\Big)
  \delta_F\Big(-\frac{1-y}{y^2 \ddot \lambda} -1\Big)
  \frac{\|x\|^{\dot c_0-1}}{\|1-x\|^{1-\dot a_0}}
  \frac{\|y\|^{\dot c_1-1}}{\|1-y\|^{1-\dot a_1}}
  \operatorname{d}(x,y) \,,
\end{aligned}
\ee
where the second equality comes from a change of variables $x \mapsto \ddot x/a_0$
and $y \mapsto y \ddot a_1$, the third equality comes
from integrating out the $\alpha$, $\beta$ and $\gamma$ variables
using~\eqref{deltaB} and~\eqref{deltaBF}, and the last equality follows from
the edge-balancing equations~\eqref{41balance}.

The vanishing of the arguments of the three delta functions gives the gluing equations
\be
\label{41ge}
X_\ve \,\, : \,\,
1-x = - \ve^{-1} x^{2}, \qquad 1-y = - \ve \, y^2
\ee
defining a 1-dimensional affine scheme $X$ equipped with a map $X(F) \to F^\times$
with fiber $X_\ve$ for $\ve =\ddot \lambda \in F^\times$.
Using Equations~\eqref{41ge}, the edge-balancing equations and the
meridian $\mu=a_0-a_1$, we compute that
\be
\label{41xnomrs}
\frac{\|x\|^{\dot c_0-1}}{\|1-x\|^{1-\dot a_0}}
\frac{\|y\|^{\dot c_1-1}}{\|1-y\|^{1-\dot a_1}} =
\frac{1}{\|\ddot \lambda \|^{\dot \mu}}
\| x\|^{\dot \lambda-2} \| y\|^{\dot \lambda+2} \,.
\ee

The above discussion, combined with Lemma~\eqref{lem.res}, implies that
$I_F(4_1,\theta)$ depends only on $\lambda=
(\dot \lambda, \ddot \lambda)$ and $\dot \mu$,
and is given by
\be
\label{IF41a}
\IF_{F,4_1}(\ve,s,t) = \frac{1}{\| \ve \|^t}
\sum_{(x,y) \in X_\ve(F^\times)} \frac{
\| x\|^{\dot \lambda-2} \| y\|^{\dot \lambda+2}}{\|\Jac(f(x,y)) \|}
\ee
where $\ve=\ddot \lambda \in F^\times$, $s=\dot \lambda \in \BR$ and
$t=\dot \mu \in \BR$ and $f=(f_1,f_2)$ where each $f_i$ is the argument of
the delta function equations. 

\subsection{The $5_2$ knot}
\label{sub.52}


Consider the triangulation of the complement of the $5_2$ knot with
three positively oriented tetrahedra $T_j$ for $j=0,1,2$ with the edge and
face-pairings given by
\begin{equation}
\label{52faces}
\begin{array}{c|cccccc|}
\hbox{\diagbox[width=2.5em]{\tiny{tet}}{\tiny{edge}}} &  01 & 02 & 03 & 12 & 13 & 23 \\
\hline
    0  &            0 &  1 &  1 &  0 &  2 &  2 \\  
    1  &            2 &  0 &  1 &  1 &  1 &  2 \\
    2  &            0 &  2 &  1 &  2 &  0 &  1 \\
  \end{array}
  \qquad\qquad
\begin{array}{c|cccc|} 
\hbox{\diagbox[width=2.5em]{\tiny{tet}}{\tiny{face}}} & 012 & 013 & 023 & 123 \\ \hline
    0  &             0 &  1 &  2 &  3  \\
    1  &             4 &  5 &  1 &  2  \\
    2  &             3 &  0 &  5 &  4  \\
 \end{array}
  \end{equation}
Labeling each face by a variable $x_j$ for $j=0,\dots,5$, 
the delta function equations of the kinematical kernel are
\begin{align*}
  -x_0 + x_1 + z_0&=0,
  & -x_4 + x_5 + z_1&=0,
  & x_0 - x_3 + z_2&=0, \\
  x_1 - x_2 + x_3&=0,
& -x_1 + x_2 + x_5&=0,  
  & x_0 + x_4 - x_5&=0,
\end{align*}
with unique solution
$$
(x_0,x_1,x_2,x_4,x_4,x_5)=
(-z_1, -z_0 - z_1, -z_0 - 2 z_1 + z_2, -z_1 + z_2, 2 z_1 - z_2, z_1 - z_2) \,.
$$
The kinematical kernel is
\be
\label{52kin}
K_X(z)=\prod_{0 \leq i, j \leq 2} \lb {z_i}{z_j}^{\frac{1}{2}Q_{ij}}, \qquad
Q=\begin{pmatrix}
  0 & -2 & 1 \\
  -2 & -4 & 3 \\
  1 & 3 & -2
  \end{pmatrix} \,.
\qquad
\ee
Writing
$$
z_0=(\alpha,x), \qquad z_1=(\beta,y), \qquad z_2=(\gamma,z)
$$
we obtain that
$$
K_X(z)=\alpha(y^{-2}z) \, \beta(x^{-2} y^{-4}z^3 ) \, \gamma(xy^3z^{-2}) \,.
$$
The balancing condition on two edges edges is
\be
\label{52balance}
b_0 + c_0 + b_1 + 2 c_1 + a_2 + c_2 =(2,1), \qquad
a_0 + b_0 + 2 a_1 + b_2 + c_2 = (2,1) 
\ee
from which we can express all angles in terms of $a_0$, $a_1$, $a_2$ and $c_1$.
After balancing the angles at all edges, the angle holonomy of half of the
longitude is given by
\be
\label{52lon}
\lambda = 2 a_0 + 4 a_1 - 3 a_2 - c_1 \,.
\ee

The integral is given by
\be
\label{Z52b}
\begin{aligned}
\IF_{F,X,\theta}= & \int_{\hat\sB^3\times\sB^3}
  \alpha\Big(\frac{z}{(1-\ddot a_0x)\ddot c_0 y^2}\Big)
  \beta\Big(\frac{z^3}{(1-\ddot a_1y)\ddot c_1x^2y^4}\Big)
  \gamma\Big(\frac{x y^3}{(1-\ddot a_2y)\ddot c_2z^2}\Big) \\
  & \times 
  \frac{\|\ddot a_0 x\|^{\dot c_0}}{\|1-\ddot a_0x\|^{1-\dot a_0}}
  \frac{\|\ddot a_1 y\|^{\dot c_1}}{\|1-\ddot a_1y\|^{1-\dot a_1}}
  \frac{\|\ddot a_2 z\|^{\dot c_2}}{\|1-\ddot a_2z\|^{1-\dot a_2}}
  \operatorname{d}(\alpha,\beta,\gamma,x,y,z) \\
 =& \int_{\hat\sB^3\times\sB^3}
  \alpha\Big(\frac{\ddot c_0^{-1} \ddot a_1^2 \ddot a_2^{-1} z}{(1-x)y^2}\Big)
  \beta\Big(\frac{\ddot c_1^{-1} \ddot a_0^2 \ddot a_1^4 \ddot a_2^{-3} z^3}{
    (1-y)x^2y^4}\Big)
  \gamma\Big(\frac{\ddot c_2^{-1} \ddot a_0^{-1} \ddot a_1^{-3} \ddot a_2^2 x y^3}{
    (1-y)z^2}\Big) \\
  & \times 
  \frac{\|x\|^{\dot c_0}}{\|1-x\|^{1-\dot a_0}}
  \frac{\|y\|^{\dot c_1}}{\|1-y\|^{1-\dot a_1}}
  \frac{\|z\|^{\dot c_2}}{\|1-z\|^{1-\dot a_2}}
  \operatorname{d}(\alpha,\beta,\gamma,x,y,z) \\
  \stackrel{\eqref{deltaB},\eqref{deltaBF}}{=}   & \int_{\sB^3}
  \delta_F\Big(\frac{\ddot c_0^{-1} \ddot a_1^2 \ddot a_2^{-1} z}{(1-x)y^2}-1\Big)
  \delta_F\Big(\frac{\ddot c_1^{-1} \ddot a_0^2 \ddot a_1^4 \ddot a_2^{-3} z^3}{
    (1-y)x^2y^4}-1\Big)
    \delta_F\Big(\frac{\ddot c_2^{-1} \ddot a_0^{-1} \ddot a_1^{-3} \ddot a_2^2 x y^3}{
    (1-y)z^2}-1\Big) \\
  & \times 
  \frac{\|x\|^{\dot c_0}}{\|1-x\|^{1-\dot a_0}}
  \frac{\|y\|^{\dot c_1}}{\|1-y\|^{1-\dot a_1}}
  \frac{\|z\|^{\dot c_2}}{\|1-z\|^{1-\dot a_2}}
  \operatorname{d}(x,y,z) \\
  \stackrel{\eqref{2haar}}=&\int_{F^3}
  \delta_F\Big(\frac{\ddot c_0^{-1} \ddot a_1^2 \ddot a_2^{-1} z}{(1-x)y^2}-1\Big)
  \delta_F\Big(\frac{\ddot c_1^{-1} \ddot a_0^2 \ddot a_1^4 \ddot a_2^{-3} z^3}{
    (1-y)x^2y^4}-1\Big)
    \delta_F\Big(\frac{\ddot c_2^{-1} \ddot a_0^{-1} \ddot a_1^{-3} \ddot a_2^2 x y^3}{
    (1-y)z^2}-1\Big) \\
  & \times 
  \frac{\|x\|^{\dot c_0-1}}{\|1-x\|^{1-\dot a_0}}
  \frac{\|y\|^{\dot c_1-1}}{\|1-y\|^{1-\dot a_1}}
  \frac{\|z\|^{\dot c_2-1}}{\|1-z\|^{1-\dot a_2}}
  \operatorname{d}(x,y,z) \,.
\end{aligned}
\ee
where the second equality comes from a change of variables $x \mapsto \ddot x/a_0$,
$y \mapsto y/\ddot a_1$ and $z \mapsto z/\ddot a_2$ and the third equality comes
from integrating out the $\alpha$, $\beta$ and $\gamma$ variables
using~\eqref{deltaB} and~\eqref{deltaBF}.  

The $F^\times$-component of the edge-balancing equations~\eqref{52balance} implies that
\be
\label{dd52}
(\ddot c_0^{-1} \ddot a_1^2 \ddot a_2^{-1},
\ddot c_1^{-1} \ddot a_0^2 \ddot a_1^4 \ddot a_2^{-3},
\ddot c_2^{-1} \ddot a_0^{-1} \ddot a_1^{-3} \ddot a_2^2) =
(1, \ddot \lambda , \ddot \lambda^{-1}) \,.
\ee
This means that the vanishing of the arguments of the three delta functions
gives the gluing equations
\be
\label{52ge}
X_\ve \,\, : \,\,
1-x = y^{-2} z, \qquad 1-y = \ve \, x^{-2} y^{-4} z^3, \qquad
1-z=\ve^{-1} \, x y^3 z^{-2}   
\ee
definining a 1-dimensional affine scheme $X$ equipped with a map $X(F) \to F^\times$
with fiber $X_\ve$ for $\ve =\ddot \lambda \in F^\times$.

On the other hand, 
the $\BR$-component of the edge-balancing equations~\eqref{52balance} implies that
\be
\label{52Qddd}
-(\dot c_0, \dot c_1, \dot c_2) - (\dot a_0,\dot a_1,\dot a_2) Q =
\dot \lambda (0, 1, -1) \,.
\ee
This, combined with Equations~\eqref{52ge}, implies that
\be
\label{52xnorms}
\frac{\|x\|^{\dot c_0-1}}{\|1-x\|^{1-\dot a_0}}
  \frac{\|y\|^{\dot c_1-1}}{\|1-y\|^{1-\dot a_1}}
  \frac{\|z\|^{\dot c_2-1}}{\|1-z\|^{1-\dot a_2}}
  = \frac{1}{\| \ve \|^{\dot \mu}} \|x\|^{e_0} \|y\|^{e_1} \|z\|^{e_2} 
\ee
where
\be
\label{d52}
\begin{aligned}
(e_0,e_1,e_2) &:= (\dot c_0-1, \dot c_1-1, \dot c_2-1)
- (1-\dot a_0, 1-\dot a_1, 1-\dot a_2) Q  \\
&= (\dot c_0, \dot c_1, \dot c_2)+(\dot a_0, \dot a_1, \dot a_2) Q
-(1,1,1)-Q.(1,1,1) \\ &= \dot \lambda(0,1,-1) +(0,2,-3)  
\end{aligned}
\ee
and
\be
\label{dmer52}
\dot\th_\mu=(0,1,-1).(1-\dot a_0, 1-\dot a_1,1-\dot a_2) = -\dot a_1 + \dot a_2 \,.
\ee
The above discussion, combined with Lemma~\eqref{lem.res}, implies that
$I_F(5_2,\theta)$ depends only on $\lambda=
(\dot \lambda, \ddot \lambda)$ and $\dot \mu$,
and is given by
\be
\label{IF52a}
\IF_{F,5_2}(\ve,s,t) = \frac{1}{\| \ve \|^t}
\sum_{(x,y,z) \in X_\ve(F^\times)} \frac{ \|x\|^{s+2} \|y\|^{-s-3}
}{\|\Jac(f(x,y,z)) \|}
\ee
where $\ve =\ddot \lambda \in F^\times$, $s=\dot \lambda \in \BR$ and
$t=\dot \mu \in \BR$.

Let us comment on the geometry of the 1-dimensional scheme $X$ defined by
Equations~\eqref{52ge}. The projection map $X(F) \to F^\times$ which sends
$(x,y,z,\ve)$ to $\ve$ has degree 7, equal to the degree of the $A$-polynomial of
$5_2$ with respect to $M$. In fact, the map $X(F) \to F^\times$ factors through
a map $X(F) \to X_A(F) \stackrel{L}{\to} F^\times$ where $X(F) \to X_A(F)$
is a birational map, $A(M,L)$ is the $A$-polynomial of $5_2$ and $X_A$ is
the corresponding curve.


\subsection{The $(-2,3,7)$ pretzel knot}
\label{sub.237}


We next discuss the case of the $(-2,3,7)$ pretzel knot. Although this knot
has is scissors congruent and the same cubic trace field as the $5_2$ knot, its
character variety is more interesting. Consider the triangulation of the complement
of the $(-2,3,7)$ pretzel knot with isometry signature \texttt{eLAkaccddjgnqw}.
It has four positively oriented tetrahedra $T_j$ for $j=0,1,2,4$ with the edge and
face-pairings given by
\begin{equation}
\label{237faces}
\begin{array}{c|cccccc|}
\hbox{\diagbox[width=2.5em]{\tiny{tet}}{\tiny{edge}}} &  01 & 02 & 03 & 12 & 13 & 23 \\
\hline
    0  &            0 & 1 & 2 & 0 & 1 & 0 \\
    1  &            3 & 0 & 2 & 1 & 3 & 1 \\
    2  &            1 & 3 & 2 & 1 & 0 & 3 \\
    3  &            3 & 1 & 0 & 3 & 1 & 3 \\
  \end{array}
  \qquad\qquad
\begin{array}{c|cccc|} 
\hbox{\diagbox[width=2.5em]{\tiny{tet}}{\tiny{face}}} & 012 & 013 & 023 & 123 \\
\hline
    0  &             0 & 1 & 2 & 3 \\
    1  &             3 & 4 & 1 & 5 \\
    2  &             5 & 2 & 4 & 6 \\
    3  &             7 & 3 & 6 & 7 \\
\end{array}
\end{equation}

Using this data, and following the steps of the computation as was done
in Section~\ref{sub.52} for the $5_2$ knot, we find that the kinematical kernel is
\be
\label{237kin}
K_X(z)=\prod_{0 \leq i, j \leq 3} \lb {z_i}{z_j}^{\frac{1}{2}Q_{ij}}, \qquad
Q=\begin{pmatrix}
  2 & -1 & -1 & -2 \\
  -1 & -8 & 9 & 1 \\
  -1 & 9 & -8 & 1 \\
  -2 & 1 & 1 & 4
  \end{pmatrix} \,.
\qquad
\ee
The edge-balancing equations (with one edge removed) are
\be
\label{237balance}
\begin{aligned}
  2 a_0 + b_1 + b_2 + c_0 + c_3 &=(2,1), \\
  a_1 + a_2 + 2 b_0 + 2 b_3 + c_1 + c_2 &=(2,1), \\
  c_0 + c_1 + c_2 &=(2,1),
\end{aligned}
\ee
which may be solved to express all angles in terms of $a_0$, $a_1$, $a_2$, $a_3$
and $c_1$. After balancing the angles at all edges, the angle holonomy of half of
the longitude is given by
\be
\label{237lon}
\lambda = a_0 + 8 a_1 - 9 a_2 - a_3 - c_1 \,.
\ee
The $F^\times$-component of the edge-balancing equations~\eqref{237balance} implies that
\be
\label{dd237}
(\ddot c_0^{-1} \ddot a_0^{-2} \ddot a_1^{} \ddot a_2^{} \ddot a_3^{2},
\ddot c_1^{-1} \ddot a_0^{} \ddot a_1^{8} \ddot a_2^{-9} \ddot a_3^{-1},
\ddot c_2^{-1} \ddot a_0^{} \ddot a_1^{-9} \ddot a_2^{8} \ddot a_3^{-1},
\ddot c_3^{-1} \ddot a_0^{2} \ddot a_1^{-1} \ddot a_2^{-1} \ddot a_3^{-4}) =
(1, \ddot \lambda , \ddot \lambda^{-1},1) \,,
\ee
giving the gluing equations
\be
\label{237ge}
\begin{aligned}
X_\ve \,\, : \,\, & &
1-x &= x^2 y^{-1} z^{-1} w^{-2}, & 1-y &= \ve \, x^{-1} y^{-8} z^9 w, \\
& & 1-z &=\ve^{-1} \, x^{-1} y^{-9} z^{8}, & 1-w &= x^{-2} y z w^4 \,.  
\end{aligned}
\ee
On the other hand, the $\BR$-component of the edge-balancing
equations~\eqref{237balance} implies that
\be
\label{237Qddd}
-(\dot c_0, \dot c_1, \dot c_2, \dot c_3) - (\dot a_0,\dot a_1,\dot a_2,\dot a_3) Q =
\dot \lambda (0, 1, -1,0) \,.
\ee
This, combined with Equations~\eqref{237ge}, implies that
\be
\label{237xnorms}
\frac{\|x\|^{\dot c_0-1}}{\|1-x\|^{1-\dot a_0}}
  \frac{\|y\|^{\dot c_1-1}}{\|1-y\|^{1-\dot a_1}}
  \frac{\|z\|^{\dot c_2-1}}{\|1-z\|^{1-\dot a_2}}
  \frac{\|w\|^{\dot c_3-1}}{\|1-z\|^{1-\dot a_3}}
  = \frac{1}{\| \ve \|^{\dot \mu}}
  \|x\|^{} \|y\|^{-\dot \lambda -2} \|z\|^{\dot \lambda}
  \|w\|^{-3} \,.
\ee
The above, combined with ~\eqref{lem.res}, implies that
$I_F((-2,3,7),\theta)$ depends only on $\lambda=
(\dot \lambda, \ddot \lambda)$ and $\dot \mu$,
and is given by
\be
\label{IF237a}
\IF_{F,(-2,3,7)}(\ve,s,t) = \frac{1}{\| \ve \|^t}
\sum_{(x,y,z,w) \in X_\ve(F^\times)} \frac{\|x\|^{} \|y\|^{-s -2} \|z\|^{s}
  \|w\|^{-3}}{\|\Jac(f(x,y,z,w)) \|}
\ee
where $\ve =\ddot \lambda \in F^\times$, $s=\dot \lambda \in \BR$ and
$t=\dot \mu \in \BR$.

The geometry of the 1-dimensional scheme $X$ defined by
Equations~\eqref{237ge} is the following. The projection map $X(F) \to F^\times$ 
which sends $(x,y,z,w,\ve)$ to $\ve$ has degree 55, equal to that of
the $A$-polynomial of the $(-2,3,7)$ knot with respect to the meridian.
and in fact, the map $X \to F^\times$ factors through
a map $X(F) \to X_A(F) \stackrel{L}{\to} F^\times$ where $A(M,L)$ is the
$A$-polynomial of $(-2,3,7)$ pretzel knot and $X_A$ is the corresponding curve.


\subsection{Point counts of zero-dimensional schemes} 
\label{sub.count}

The distributional invariant $\IF_F$ computed above leads, after evaluation
at a test function, to point-counts of zero-dimensional schemes. 
In this section we list three elementary facts about point counts of reduced
zero-dimensional schemes $X$ which were pointed out to us by Frank Calegari.

Suppose $X$ is a reduced 0-dimensional scheme over $\BQ$. Then, there exists
a finite set $S$ of rational primes such that for all local fields $F$ with
discrete valuation ring $\calO_F$ and residue field $\BF_q$ where $q$ is a power
of a prime not in $S$, we have:

\be
\label{count1}
|X(F)| =|X(\calO_F)| = |X(\BF_q)| \,.
\ee
The first equality follows from the fact that $X$ can be spread as a 0-dimensional
scheme on $\BZ[1/S]$, and the second follows from Hensel's lemma. 

The second fact is that the point counts $|X(\BF_p)|$ for $p \not\in S$
determine the point count $X(\overline\BQ)$. This follows from Chebotarev
density theorem. 

A third fact is that when a component $Y(\overline\BQ)$ of $X(\overline\BQ)$
is defined over a Galois field $K$ which is disjoint from that of the other
components, then we can find a positive density set of primes such that
\be
\label{XC}
|X(\BQ_p)| = |Y(\BQ_p)| \,,
\ee
i.e., the $\BQ_p$-count on $X$ equals to that of the $Y$ and it is isolated from
that of $X\setminus Y$. 


\section{A pair of edge-type generalized TQFTs}
\label{sec.TQFT2}

In the previous Section~\ref{sec.TQFT1} we constructed a face-type generalized
TQFT using a quantum dilogarithm $\vphi$~\eqref{QDL} on the Gaussian group
$\sA=\hat\sB\times \sB$ (with $\sB=F^\times$) associated to a local field $F$.
In this section we construct a pair of edge-type generalized TQFTs using the
Weil transform of $\vphi$, one with respect to $\hat\sB$ and another with
respect to $\sB$. In favorable
circumstances each of these Weil transforms leads to a edge-type generalized TQFT
whose states are placed in the edges of the tetrahedra (and not in the faces, as
was the case of Section~\ref{sec.TQFT1}). When additional symmetries are found,
the triangulations are unordered (yet oriented) and the weights of the tetrahedra
manifestly depend only on the combinatorial information of the triangulation
encoded by the Neumann--Zagier matrices~\cite{NZ}. These edge-type generalized
TQFTs are sometimes called of Turaev--Viro type, because their partition functions
are edge state-integrals whose state-variables are on the edges of the ideal
triangulation, just like the original Turaev--Viro invariants~\cite{TV}. Two examples
of such generalized TQFTs using the Gaussian groups $\BR \times \BR$ and
$\BZ \times S^1$ from Examples~\ref{ex.KLV} and~\ref{ex.3Dindex} are described
in~\cite{KLV} and~\cite{GK:mero}, respectively.

Our goal in this section is to define, in addition to the face-type generalized
TQFT of Section~\ref{sec.TQFT1}, two more edge-type generalized TQFTs using the
two Weil transformations.

\subsection{The $\hat{\mathsf{B}}$-Weil transform}
\label{sub.weil}

In this section we compute the $\hat{\sB}$-Weil transform of the quantum dilogarithm
~\eqref{QDLabc} following the Appendix B of~\cite{GK:mero}. It is the 
function $g_{a,c} : \sB^2 \to \BC$ defined by 
\begin{equation}
\label{gweil}
g_{a,c}(x,z)=g_{a,c}((\alpha,x),(\gamma,z)):= \alpha(z)
\int_{\hat\sB}
\bar\Psi_{a,c}(-\alpha-\beta,1/x)\langle (\beta,1);(\gamma,z)
\rangle\operatorname{d}\!\beta
\end{equation}
where the independence of the above definition from $\alpha$ and $\gamma$ is
assured from the invariance properties of the Weil transformation
(see Appendix B of~\cite{GK:mero} for details).

\begin{lemma}
\label{lem.weil1}
We have:
\be
\label{eq.weil1}
g_{a,c}(x,z) =f_{\dot a,\dot c}(1/(x\ddot a),z\ddot c),\quad 
f_{\dot a,\dot c}(x,z):=\|x\|^{\dot c} \|z\|^{\dot a}\delta_{F}(x+z-1) \,.
\ee
The function $g_{a,c}$ satisfies the symmetries
\begin{equation}
\label{eq:cyc-sym-gac}
g_{a,c}(x,z)=g_{b,a}(y,x)=g_{c,b}(z,y),\quad  a+b+c=\varpi,\quad xyz=1.
\end{equation}
\end{lemma}

\begin{proof}
By using the formula 
$$
  \bar\Psi_{a,c}(\alpha,x)=\alpha\big((1-x/\ddot a)/\ddot c\big)
  \frac{\|x/\ddot a\|^{\dot c}}{\|1-x/\ddot a\|^{1-\dot a}},
$$
we compute:  
\begin{align*}
  g_{a,c}(x,z)&=\alpha(z)\int_{\hat\sB}
 (-\alpha-\beta)\big((1-(x\ddot a)^{-1})/\ddot c\big)
 \frac{\|1/(x\ddot a)\|^{\dot c}}{\|1-(x\ddot a)^{-1}\|^{1-\dot a}}
 \beta(z)\operatorname{d}\!\beta \\
  &=\int_{\hat\sB}(\alpha+\beta)\Big(\frac{z\ddot c}{1-(x\ddot a)^{-1}}\Big)
  \frac{\|x\ddot a\|^{-\dot c}}{\|1-(x\ddot a)^{-1}\|^{1-\dot a}}\operatorname{d}\!\beta 
  =\delta_{\sB}\Big(\frac{z\ddot c}{1-(x\ddot a)^{-1}}\Big)
  \frac{\|x\ddot a\|^{-\dot c}}{\|1-(x\ddot a)^{-1}\|^{1-\dot a}}\\
  &=\delta_{F}\Big(\frac{z\ddot c}{1-(x\ddot a)^{-1}}-1\Big)
  \frac{\|x\ddot a\|^{-\dot c}}{\|z\ddot c\|^{1-\dot a}}
   =
   \frac{\|z\ddot c\|^{\dot a}}{\|x\ddot a\|^{\dot c}}\delta_{F}\Big(z\ddot c
   +\frac1{x\ddot a}-1\Big)=f_{\dot a,\dot c}(1/(x\ddot a),z\ddot c) \,.
\end{align*}
This concludes the proof of Equation~\eqref{eq.weil1}. The rest follows
easily from part this equation using part (a) of Lemma~\ref{lem.res0}.
\end{proof}

\subsection{The $\mathsf{B}$-edge TQFT}
\label{sub.BTQFT}


Based on the function $g_{a,c}(x,z)$, one can formulate a generalized TQFT model of the
Turaev--Viro type (like the ones in~\cite{KLV} and~\cite{GK:mero}) using
\emph{unordered} triangulations, and placing integration variables (lying in
$F^\times$) at the edges of triangulations (as opposed to the faces and tetrahedra
done in the previous section), with the angle-dependent symmetric tetrahedral
Boltzmann weights of a tetrahedron $T$ obtained under the substitutions
$x\to \frac{x_{0,2}x_{1,3}}{x_{0,3}x_{1,2}}$ and
$z\to \frac{x_{0,1}x_{2,3}}{x_{0,2}x_{1,3}}$ into $g_{a,c}(x,z)$.
This results into the tetrahedral weight

\begin{equation}
  \label{Tweil1}
  W_{a,c}(T,x)
  = \|x_{0,1}x_{2,3}\ddot c\|^{\dot a}\|x_{0,2}x_{1,3}\|^{\dot b}
  \|x_{0,3}x_{1,2}/\ddot a\|^{\dot c}\delta_{F}\Big(x_{0,1}x_{2,3}\ddot c
  +\frac{x_{0,3}x_{1,2}}{\ddot a}-x_{0,2}x_{1,3}\Big)
\end{equation}
where $x_{i,j}$ is the  $\sB$-valued state variable on the geometric edge
opposite to the edge $\partial_i\partial_j T$ of the tetrahedron $T$. 
Indeed, when $z=\frac{x_{0,1}x_{2,3}}{x_{0,2}x_{1,3}}$, we have:
\begin{align*}
  \frac{\|z\ddot c\|^{\dot a}}{\|x\ddot a\|^{\dot c}}\delta_{F}
  \Big(z\ddot c+\frac1{x\ddot a}-1\Big) &=
  \frac{\|\frac{x_{0,1}x_{2,3}}{x_{0,2}x_{1,3}}\ddot c\|^{\dot a}}{
    \|\frac{x_{0,2}x_{1,3}}{x_{0,3}x_{1,2}}\ddot a\|^{\dot c}}
  \delta_{F}\Big(\frac{x_{0,1}x_{2,3}}{x_{0,2}x_{1,3}}\ddot c
  +\frac1{\frac{x_{0,2}x_{1,3}}{x_{0,3}x_{1,2}}\ddot a}-1\Big)\\
  &= \frac{\|x_{0,1}x_{2,3}\ddot c\|^{\dot a}\|x_{0,2}x_{1,3}\|^{1-\dot a}}{
    \|x_{0,2}x_{1,3}\ddot a\|^{\dot c}\|x_{0,3}x_{1,2}\|^{-\dot c}}
  \delta_{F}\Big(x_{0,1}x_{2,3}\ddot c+\frac{x_{0,3}x_{1,2}}{\ddot a}-x_{0,2}x_{1,3}\Big)
  \\
&= \|x_{0,1}x_{2,3}\ddot c\|^{\dot a}\|x_{0,2}x_{1,3}\|^{\dot b}
\|x_{0,3}x_{1,2}/\ddot a\|^{\dot c}\delta_{F}\Big(x_{0,1}x_{2,3}\ddot c
+\frac{x_{0,3}x_{1,2}}{\ddot a}-x_{0,2}x_{1,3}\Big) \\
&= \|x_{0,2}x_{1,3}\ddot a\|^{\dot b}\|x_{0,3}x_{1,2}\|^{\dot c}
\|x_{0,1}x_{2,3}/\ddot b\|^{\dot a}\delta_{F}\Big(x_{0,2}x_{1,3}\ddot a
+\frac{x_{0,1}x_{2,3}}{\ddot b}-x_{0,3}x_{1,2}\Big)  
\end{align*}
where the last equality reflects the invariance of the weight under the cyclic
permutation $(1,a)\mapsto (2,b)\mapsto (3,c)\mapsto(1,a)$.

This results into a face-type generalized TQFT whose partition function we denote by
$\IE_{\sB}$. 

\begin{remark}
\label{rem.zickertA}
The arguments in the delta functions in~\eqref{Tweil1} are very similar to the
defining equations of Zickert's enhanced Ptolemy variety of an ideal triangulation
$\calT$, but modified by the angle data; see~\cite{Zickert:Apoly}. This fact can be
the key point for explaining the possible relation between the generalized TQFT
invariant of a 3-manifold constructed using the tetrahedral weight~\eqref{Tweil1}
and the $A$-polynomial of a knot~\cite{CCGLS}. We will illustrate this with the
example of the $4_1$ knot.
\end{remark}

\begin{example}
\label{ex.41g}
The partition function of the $\sB$-edge generalized TQFT for the ideal triangulation
of the complement of the figure eight knot with two ideal tetrahedra
is given by the integral
\begin{equation}
\IE_{\sB,4_1}(\lambda,\mu)=\int_{\sB} g_{a_0,c_0}(x, 1/x^2)
\bar g_{a_1,c_1}(1/x, x^2)\operatorname{d}\!x
\end{equation}
which we can calculate by using~\eqref{eq.weil1}, the balancing condition
$2a_0+c_0=2a_1+c_1$ and the definitions of the longitude $\lambda=a_0-b_0$ and the
meridian $\mu=a_0-a_1$
\begin{align*}
\IE_{\sB,4_1}(\lambda,\mu)
&=\int_{\sB} f_{\dot a_0,\dot c_0}\big(\tfrac1{x\ddot a_0}, \tfrac{\ddot c_0}{x^2}\big)
f_{\dot a_1,\dot c_1}\big(x\ddot a_1, \tfrac{x^2}{\ddot c_1}\big)
\operatorname{d}\!x
=\int_{\sB} f_{\dot a_0,\dot c_0}\big(\tfrac{1}{x}, \tfrac{\ddot c_0\ddot a_0^2}{x^2}\big)
f_{\dot a_1,\dot c_1}
\big(\tfrac{x\ddot a_1}{\ddot a_0}, \tfrac{x^2}{\ddot c_1\ddot a_0^2}\big)
\operatorname{d}\!x\\
&=\int_{\sB} f_{\dot a_0,\dot c_0}\big(\tfrac{1}{x}, -\tfrac{\ddot \lambda}{x^2}\big)
f_{\dot a_1,\dot c_1}\big(\tfrac{x}{\ddot \mu}, -\tfrac{x^2}{\ddot \lambda\ddot \mu^2}
\big)\operatorname{d}\!x \\
&=\int_{\sB}\frac{\|\ddot\lambda\|^{\dot a_0-\dot a_1}\|x\|^{2\dot a_1+\dot c_1}}{
\|\ddot\mu\|^{2\dot a_1+\dot c_1}\|x\|^{2\dot a_0+\dot c_0}}
\delta_\sB\big(\tfrac{1}{x} -\tfrac{\ddot \lambda}{x^2}\big)
\delta_\sB\big(\tfrac{x}{\ddot \mu} -\tfrac{x^2}{\ddot \lambda\ddot \mu^2}\big)
\operatorname{d}\!x\\
&=\int_{F}\frac{\|\ddot\lambda\|^{\dot \mu}\|x\|}{\|\ddot\mu\|^{\dot\lambda}}
\delta_F\big(x -\ddot \lambda-x^2\big)
\delta_F\big(x -\tfrac{x^2}{\ddot \lambda\ddot \mu}-\ddot\mu\big)
\operatorname{d}\!x\\
&=\int_{F}\frac{\|\ddot\lambda\|^{\dot \mu}\|x\|}{\|\ddot\mu\|^{\dot\lambda}}
\delta_F\big(x -\ddot \lambda-x^2\big)
\delta_F\big(x -\tfrac{x-\ddot \lambda}{\ddot \lambda\ddot \mu}-\ddot \mu\big)
\operatorname{d}\!x\\
&=\frac{\|\ddot\lambda\|^{\dot \mu}\|\ddot\mu-1/\ddot\mu\|}{\|\ddot\mu\|^{\dot\lambda}}
\delta_F\Big((\ddot\mu -\ddot \lambda)(1-\tfrac1{\ddot \lambda\ddot \mu})
-(\ddot\mu-1/\ddot\mu)^2\Big)
=\frac{\|\ddot\lambda\|^{\dot \mu}}{\|\ddot\mu\|^{\dot\lambda}}
\|\ddot\mu-1/\ddot\mu\|\delta_F\big(A_{4_1}(\ddot\lambda,\ddot\mu)\big)
\end{align*}
where
\begin{equation}
\label{A41}
A_{4_1}(L,M):=L+L^{-1}+(M-M^{-1})^2-M-M^{-1}
\end{equation}
is the $A$-polynomial of the figure-eight knot~\cite{CCGLS}.
\end{example}


\subsection{The $\mathsf{B}$-Weil transform}

We now consider the $\sB$-Weil transform of the quantum dilogarithm
~\eqref{QDLabc}. It is the function $h_{a,c} : \hat\sB^2 \to \BC$ defined by 
\begin{equation}
\label{hweil}
h_{a,c}(\alpha,\gamma)=h_{a,c}((\alpha,x),(\gamma,z)):= \gamma(x)
\int_{\sB}\bar\Psi_{a,c}(-\alpha,1/(xy))\langle (0,y);(\gamma,z)
\rangle\operatorname{d}\!y \,.
\end{equation}

\begin{lemma}
\label{lem.weil2}
We have:
\be
\label{eq.weil2}
h_{a,c}(\alpha,\gamma) = \int_{\sB}\alpha\Big(\frac{\ddot c}{1-y}\Big)
\frac{\|y\|^{\dot c}}{\|1-y\|^{1-\dot a}\gamma(y\ddot a)}\operatorname{d}\!y.
\ee
\end{lemma}

\begin{proof}
We compute:
\begin{multline*}
h_{a,c}(\alpha,\gamma) =\gamma(x)
\int_{\sB}\bar\Psi_{a,c}(-\alpha,1/(xy))\gamma(y)\operatorname{d}\!y=\gamma(x)
\int_{\sB}\bar\Psi_{a,c}(-\alpha,1/y)\gamma(y/x)\operatorname{d}\!y\\
=
\int_{\sB}\bar\Psi_{a,c}(-\alpha,1/y)\gamma(y)\operatorname{d}\!y
=
\int_{\sB}\bar\Psi_{a,c}(-\alpha,y)\gamma(1/y)\operatorname{d}\!y\\
=
\int_{\sB}\alpha\Big(\frac{\ddot c}{1-y/\ddot a}\Big)
\frac{\|y/\ddot a\|^{\dot c}}{\|1-y/\ddot a\|^{1-\dot a}}\gamma(1/y)
\operatorname{d}\!y
=
\int_{\sB}\alpha\Big(\frac{\ddot c}{1-y}\Big)
\frac{\|y\|^{\dot c}}{\|1-y\|^{1-\dot a}\gamma(y\ddot a)}\operatorname{d}\!y
\end{multline*}
\end{proof}

\begin{lemma}
\label{lem.gh}
\rm{(a)} The functions $g_{a,c}(x,z)$ are $h_{a,c}(\alpha,\gamma)$ are related
to each other by a Fourier transformation
\begin{equation}
\label{eq:hac-gac}
h_{a,c}(\alpha,\gamma)=
\int_{\sB^2}\frac{\gamma(x)}{\alpha(z)}g_{a,c}(x,z)\operatorname{d}(x,z) \,,
\end{equation}
which corresponds to a duality symmetry of the beta pentagon
relations~\cite{Kashaev:beta}.
\newline
\rm{(b)} The function $h_{a,c}$ satisfies the symmetries
\begin{equation}
\label{eq:cyc-sym-hac}
h_{a,c}(\alpha,\gamma)=h_{b,a}(\beta,\alpha)=h_{c,b}(\gamma,\beta),
\quad  a+b+c=\varpi,\quad \alpha+\beta+\gamma=0.
\end{equation}
\end{lemma}

\begin{proof}
From~\eqref{gweil} we have
\begin{equation}
g_{a,c}(x,z)=
\int_{\hat\sB}
\bar\Psi_{a,c}(-\beta,1/x)\beta(z)\operatorname{d}\!\beta
\quad\Leftrightarrow\quad
\bar\Psi_{a,c}(-\alpha,1/x)=\int_{\sB}\frac{g_{a,c}(x,z)}{\alpha(z)}
\operatorname{d}\! z
\end{equation}
and from~\eqref{hweil}
\begin{equation}
h_{a,c}(\alpha,\gamma)= 
\int_{\sB}\bar\Psi_{a,c}(-\alpha,1/y)\gamma(y)\operatorname{d}\!y
\quad\Leftrightarrow\quad \bar\Psi_{a,c}(-\alpha,1/x)
=\int_{\hat\sB}\frac{h_{a,c}(\alpha,\gamma)}{\gamma(x)}\operatorname{d}\! \gamma
\end{equation}
so that 
\begin{equation}
\int_{\sB}\frac{g_{a,c}(x,z)}{\alpha(z)}\operatorname{d}\! z
=\int_{\hat\sB}\frac{h_{a,c}(\alpha,\gamma)}{\gamma(x)}\operatorname{d}\!
\gamma\quad\Leftrightarrow\quad h_{a,c}(\alpha,\gamma)
=\int_{\sB^2}\frac{\gamma(x)}{\alpha(z)}g_{a,c}(x,z)\operatorname{d}(x,z).
\end{equation}
To show the cyclic symmetries~\eqref{eq:cyc-sym-hac}, we compute
\begin{align*}
  h_{a,c}(\alpha,\gamma)&=
  \int_{\sB^2}\frac{\gamma(x)}{\alpha(z)}g_{a,c}(x,z)
\operatorname{d}(x,z)=\int_{\sB^2}\frac{\gamma(x)}{\alpha(z)}g_{b,a}((xz)^{-1},x)
\operatorname{d}(x,z)\\
&=\int_{\sB^2}\frac{\gamma(x)}{\alpha(u/x)}g_{b,a}(1/u,x)\operatorname{d}(x,u)
=\int_{\sB^2}\frac{(\alpha+\gamma)(x)}{\alpha(u)}g_{b,a}(1/u,x)
\operatorname{d}(x,u)\\
&=\int_{\sB^2}\frac{\alpha(y)}{\beta(x)}g_{b,a}(y,x)\operatorname{d}(x,y)
=h_{b,a}(\beta,\alpha) \,.
\end{align*}
\end{proof}

\subsection{The $\hat{\mathsf{B}}$-edge generalized TQFT}
\label{sub.BhatTQFT}

Based on the function $h_{a,c}(\alpha,\gamma)$, one can formulate a second
face-type generalized TQFT using \emph{unordered} triangulations, and placing
$\hat\sB$-integration variables at the edges of the triangulations with
the angle-dependent symmetric tetrahedral Boltzmann weights of a tetrahedron $T$
obtained under the substitutions
$\alpha\to \alpha_{0,2}+\alpha_{1,3}-\alpha_{0,3}-\alpha_{1,2}$ and
$\gamma\to \alpha_{0,1}+\alpha_{2,3}-\alpha_{0,2}-\alpha_{1,3}$ into
$h_{a,c}(\alpha,\gamma)$. This model whose partition function we denote by
$\IE_{\hat{\sB}}$ will be discussed in more detail for the case of the field $F=\BR$
in the next section. 


\part{The field of the real numbers}
\label{part2}

\section{The field of the real numbers}
\label{sec.R}

\subsection{The quantum dilogarithm}
\label{sub.RQDL}

In this section we discuss the quantum dilogarithm~\eqref{QDL} in more detail for
the case of the field $F=\BR$. Then, $\sB=\BR^\times$ is isomorphic to
$\BR \times \BZ/2\BZ$ by the map that sends $x$ to $(\sqrt{2}\log|x|,\sgn(x))$ where
$\sgn(x)=1$ (resp., $-1$) when $x>0$ (resp., $x<0$) and the choice of the constant
$\sqrt{2}$ is fixed by the condition of matching the Haar measures on $\BR^\times$
and on $\BR\times \BZ/2\BZ$. Since $\BR$ (and $\BZ/2\BZ$) is a self-dual LCA group, 
so is $\sB$. Thus, we can identify $\hat \sB$ with $\sB$ and denote its elements
simply by $x,y,\dots \in \BR^\times$ instead of $\alpha,\beta,\dots$.

Summarizing our discussion, when $F=\BR$, we have isomorphisms
\be
\label{Risos}
\hat\sB \simeq \sB = \BR^\times \simeq \BR \times \BZ/2\BZ \,.
\ee
For $x \in \BR^\times$, we will define $\ve_x:=0$ (resp., $1$) when $x>0$ (resp.,
$x<0$), and we will define $\ell_x:=4\pi\log|x|$. (The factor of $4\pi$ is
included here to simplify equation~\eqref{eq:hacxy} below.) Then, we have:
\be
\label{ellx}
x=(-1)^{\ve_x} e^{\frac{\ell_x}{4\pi}}, \qquad (x \in \BR^\times) \,.
\ee
Using the fact that $\ve_x=(1-\sgn(x))/2$ where $\sgn(x)=1$ for $x>0$ and
$-1$ for $x<0$, it follows that
\be
\label{ellxy}
\ve_{xy}=\ve_x+\ve_y-2\ve_x\ve_y, \qquad \ell_{xy}=\ell_x + \ell_y 
\ee
for all $x,y \in \BR^\times$. The Gaussian
exponential~\eqref{gauss} of $\sA = \BR^\times \times \BR^\times$ is given by
\be
\label{Rgauss}
\langle x,y \rangle := (-1)^{\ve_x\ve_y} e^{4 \pi i \log|x| \log|y|} 
=(-1)^{\ve_x\ve_y}  e^{\frac{i}{4\pi}\ell_x\ell_y}
\ee
and the evaluation map $\hat \sB \times \sB \to \BT$ is given by
$(y,x) \mapsto \langle y, x \rangle$. Equation~\eqref{ellxy} implies that
\be
\label{multixy}
\langle x,y \rangle= \langle y,x \rangle, \qquad
\langle x,y y' \rangle= \langle x,y \rangle \langle x,y' \rangle 
\ee
for all $x,y,y' \in \BR^\times$ and satisfies, in particular, $\langle x,1 \rangle=1$
for all $x \in \BR^\times$. With the above identifications,
the quantum dilogarithm~\eqref{QDL} and its angled version~\eqref{QDLabc}
are given by 
\begin{align}
\label{QDLR}
\vphi: \BR^\times \setminus\{-1\} \times \BR^\times \to \BT, & \qquad
\vphi(y,x) = 
(-1)^{\ve_y \ve_{1+x}}  |1+x|^{ i \ell_y} \,,
\\
\label{QDLRabc}
\Psi_{a,c}: \BR^\times\setminus\{\ddot a^{-1}\} \times \BR^\times \to \BC^\times,
& \qquad 
\Psi_{a,c}(y,x) = 
  (-1)^{\ve_y \ve_{(1-\ddot a x)\ddot c}}
  e^{-\frac{i}{4\pi}\ell_y \ell_{(1-\ddot a x)\ddot c}}
  \frac{|\ddot a x|^{\dot c}}{|1-\ddot a x|^{1-\dot a}} 
\end{align}
where $a=(\dot a ,\ddot a) \in \BR \times \BR^\times$ and likewise $c$. 
When $\dot a, \dot c>0$ with $\dot a + \dot c< 1$, then we have the bound
$\Psi_{a,c}(y,x) =O(e^{\tfrac{\ell_x}{4\pi}(\dot a+\dot c -1)})$
(resp., $O(e^{\tfrac{\ell_x}{4\pi}\dot c})$) when $\ell_x \gg 0$ (resp., $\ell_x
\ll 0$) which implies exponential decay at infinity in the $x$-direction and
boundedness in the $y$-direction. Moreover, the function $\Psi_{a,c}(y,x)$ which
has a singularity at $x=\ddot a^{-1}$ is locally integrable since it behaves
like $\ve^{\dot a-1}$ for $x=\ddot a^{-1} + \ve$. Hence, $\Psi_{a,c}(y,x)$ is
locally integrable and polynomially bounded at infinity, hence a tempered distribution
on $\sA$~\cite[Thm.V.10]{Reed-Simon:I}. 

Likewise, the functions $g_{a,c}$ and $h_{a,c}$ defined explicitly below are tempered
distributions on $\sA$ since they are partial and full Fourier transforms (see
Equations~\eqref{gweil} and~\eqref{hweil}). 


\subsection{The function $h_{a,c}$}
\label{sub.Rhac}

In this section we compute explicitly the function $h_{a,c}$ for $F=\BR$.

\begin{theorem}
\label{thm.Rweil2}
\rm{(a)} When $F=\BR$, the function $h_{a,c}$ of Equation~\eqref{eq.weil2} is given by 
\begin{equation}
\label{eq:hacxy}
\begin{aligned}
h_{a,c}(x,y)&= \frac{\langle x,\ddot c\rangle}{\langle y,\ddot a\rangle}
\big(\operatorname{B}(\dot a- i \ell_x,\dot c- i \ell_y)
+(-1)^{\ve_x}\operatorname{B}(\dot a-i \ell_x,\dot b+ i \ell_{xy})
+(-1)^{\ve_y}\operatorname{B}(\dot c- i \ell_y,\dot b+i \ell_{xy}) \big)
\\
&=
\sqrt{2\pi}\frac{\langle x,\ddot c\rangle}{\langle y,\ddot a\rangle}
(-1)^{\ve_{x}\ve_{y}}\Gamma_{\ve_{x}}(\dot a- i \ell_x)
\Gamma_{\ve_{y}}(\dot c- i \ell_y)\Gamma_{\ve_{xy}}(\dot b+i \ell_{xy})
\end{aligned}
\end{equation}
for $x,y \in \BR^\times$, where 
\begin{equation}
\label{eq:gamma_n}
\Gamma_{n}(z):=\sqrt{\frac{2}{\pi}}\Gamma(z)\cos(\pi(n-z)/2),\quad n\in\{0,1\} \,.
\end{equation}
\rm{(b)}
The function $h_{a,c}$ satisfies the symmetries
\begin{equation}
\label{eq:cyc-sym-h} 
h_{a,c}(x,z)=h_{b,a}(y,x)=h_{c,b}(z,y), \qquad a+b+c = (1,-1), \quad xyz=1 \,.
\end{equation}
\rm{(c)} 
Moreover, we have 
\be
\label{W2}
h_{a,c}(x,y) =
2\frac{\langle x,\ddot c\rangle}{\langle y,\ddot a\rangle}
\CB(\tfrac{\pi}{2}(\ve_x+i\ell_x-\dot a), \tfrac{\pi}{2}(\ve_y+i\ell_y-\dot b))
\operatorname{B}(\dot a- i \ell_x,\dot c- i \ell_y) 
\ee
for all $x,y \in \BR^\times$ where
\be
\label{CB}
\CB(x,y) := \frac{\cos(x) \cos(y)}{\cos (x+y)} 
\ee
is a trigonometric version of the beta function.  
\end{theorem}

Equation~\eqref{eq:hacxy} is similar the equation expressing the Venetziano
amplitude as a sum of three beta functions and appears in $p$-adic string theory,
see eg~\cite[Eqn.(2)]{Freund-Witten}.

\begin{proof}
For the first part, when $F=\BR$, the general formula
$$
h_{a,c}(\alpha,\gamma) =\int_{\sB}\alpha\Big(\frac{\ddot c}{1-y}\Big)
\frac{\|y\|^{\dot c}}{\|1-y\|^{1-\dot a}\gamma(y\ddot a)}\operatorname{d}_{\sB}\!y,
$$
with $\alpha=x=(-1)^{\ve_x}e^{\frac{\ell_x}{2\pi}}$ and
$\gamma=y=(-1)^{\ve_y}e^{\frac{\ell_y}{2\pi}}$ takes the form
\begin{align*}
h_{a,c}(x,y)&=\sum_{\epsilon\in\{\pm1\}}\int_0^\infty 
\frac{\big\langle x,\frac{\ddot c}{1-\epsilon u}\big\rangle
u^{\dot c-1}}{\big|1-\epsilon u\big|^{1-\dot a}
\big\langle y,\epsilon u\ddot a\big\rangle}
\operatorname{d}\!u\\
&=\frac{\langle x,\ddot c\rangle}{\langle y,\ddot a\rangle}
\sum_{k=0}^1\int_0^\infty \big|1-(-1)^ku\big|^{\dot a-1 -i \ell_x}
(-1)^{\ve_x\ve_{1-(-1)^ku}+k\ve_y}
u^{\dot c-1- i \ell_y}\operatorname{d}\!u\\
\end{align*}
so that
\begin{equation*}
\begin{aligned}
h_{a,c}(x,y) 
\langle y,\ddot a\rangle/\langle x,\ddot c\rangle & \\
&\hspace{-4cm} =\int_0^\infty |1-u|^{\dot a-1- i \ell_x}(-1)^{\ve_x\ve_{1-u}}
u^{\dot c-1-i \ell_y}\operatorname{d}\!u
+(-1)^{\ve_y}\int_1^\infty v^{\dot a-1- i \ell_x}
(v-1)^{\dot c-1-i \ell_y}\operatorname{d}\!v\\
&\hspace{-4cm} = \int_{0}^1|1-u|^{\dot a-1- i \ell_x}u^{\dot c-1- i \ell_y}
\operatorname{d}\!u
+(-1)^{\ve_x}\int_1^\infty |1-u|^{\dot a-1-i \ell_x}
u^{\dot c-1- i \ell_y}\operatorname{d}\!u\\
&\hspace{-3.5cm}
+ (-1)^{\ve_y}\int_{0}^1 t^{\dot b-1+i (\ell_x+ \ell_y)}(1-t)^{\dot c-1-i \ell_y}
\operatorname{d}\!t\\
&\hspace{-4cm} = \operatorname{B}(\dot a- i \ell_x,\dot c- i \ell_y)+ (-1)^{\ve_x}
\int_{0}^1(1-t)^{\dot a-1- i \ell_x}
t^{\dot b-1+ i (\ell_x+\ell_y)}\operatorname{d}\!t+(-1)^{\ve_y}
\operatorname{B}(\dot c- i \ell_y,\dot b+i \ell_{xy})\\
 &\hspace{-4cm} =\operatorname{B}(\dot a- i \ell_x,\dot c- i \ell_y)
+(-1)^{\ve_x}\operatorname{B}(\dot a-i \ell_x,\dot b+ i \ell_{xy})
+(-1)^{\ve_y}\operatorname{B}(\dot c- i \ell_y,\dot b+i \ell_{xy})
\end{aligned}
\end{equation*}
where
\begin{equation}
\operatorname{B}(z,w):=\int_{0}^1t^{z-1}(1-t)^{w-1}\operatorname{d}\!t.
\end{equation}
is the Euler beta function and
\begin{equation}
\Gamma(z):=\int_0^\infty t^{z-1}e^{-t}\operatorname{d}\!t.
\end{equation}
is the $\Gamma$-function. Finally, by using the formula
$$
\operatorname{B}(x,y)=\frac{\Gamma(x)\Gamma(y)}{\Gamma(x+y)}
=\Gamma(x)\Gamma(y)\Gamma(1-x-y)\frac{\sin(\pi(x+y))}{\pi},
$$
we write the singed sum of the three beta functions as a product of three
$\Gamma_n$ functions as follows:
\begin{equation*}
\begin{aligned}
\hspace{-1cm}
\operatorname{B}(\dot a- i \ell_x,\dot c- i \ell_y)
+(-1)^{\ve_x}\operatorname{B}(\dot a-i \ell_x,\dot b+ i \ell_{xy})
+(-1)^{\ve_y}\operatorname{B}(\dot c- i \ell_y,\dot b+i \ell_{xy}) & \\
&\hspace{-13cm} =\frac1{\pi}\Gamma(\dot a- i \ell_x)\Gamma(\dot c- i \ell_y)
\Gamma(\dot b+i \ell_{xy})\Big(\sin(\pi(\dot b+ i \ell_{xy}))
+(-1)^{\ve_x}\sin(\pi(\dot c- i \ell_{y})) \\
&\hspace{-12.5cm}   +(-1)^{\ve_y}\sin(\pi(\dot a- i \ell_{x}))\Big)
 =\sqrt{2\pi}(-1)^{\ve_x\ve_y}\Gamma_{\ve_x}(\dot a- i \ell_x)
 \Gamma_{\ve_y}(\dot c- i \ell_y)\Gamma_{\ve_{xy}}(\dot b+i \ell_{xy}).
\end{aligned}
\end{equation*}
This concludes the proof of the first part. The second part follows from
Equation~\eqref{eq:cyc-sym-gac}. The third part follows from the second equality
of Equation~\eqref{eq:hacxy} together with the inversion relation
\begin{equation}
\label{eq:inv-rel-sym}
\Gamma_n(z)\Gamma_n(1-z)=1, \qquad n\in\{0,1\}.
\end{equation}
of the normalized $\Gamma$-functions. 
\end{proof}
  
\subsection{A tetrahedral weight based on  the function $h_{a,c}$}

The Boltzmann weights are obtained by the substitutions
$x=\frac{x_{0,2}x_{1,3}}{x_{0,3}x_{1,2}}$ and $y=\frac{x_{0,1}x_{2,3}}{x_{0,2}x_{1,3}}$
in $h_{a,c}(x,y)$. The triangulations are now oriented but unordered, and
using Theorem~\ref{thm.Rweil2}, it follows that the tetrahedral weight of
a tetrahedron is
\be
\label{Tweil3}
\begin{aligned}
W_{a,c}(T,x) &=\sqrt{2\pi}\frac{\operatorname{G}_{\dot a}\!
\Big(\frac{x_{0,2}x_{1,3}}{x_{0,3}x_{1,2}}\Big)\operatorname{G}_{\dot b}\!
\Big(\frac{x_{0,3}x_{1,2}}{x_{0,1}x_{2,3}}\Big)\operatorname{G}_{\dot c}\!
\Big(\frac{x_{0,1}x_{2,3}}{x_{0,2}x_{1,3}}\Big)}{\langle -\ddot a,
x_{0,1}x_{2,3}\rangle\langle -\ddot b,x_{0,2}x_{1,3}\rangle\langle -\ddot c,
x_{0,3}x_{1,2}\rangle}
\\ &=
2
\frac{\CB(\tfrac{\pi}{2}(\ve_A+i\ell_A-\dot a), \tfrac{\pi}{2}(\ve_B+i\ell_B-\dot b))
\operatorname{B}(\dot a- i \ell_A,\dot b- i \ell_B)}{
\langle -\ddot a,
x_{0,1}x_{2,3}\rangle\langle -\ddot b,x_{0,2}x_{1,3}\rangle\langle -\ddot c,
x_{0,3}x_{1,2}\rangle}
\end{aligned} 
\ee
where
\begin{equation}
\operatorname{G}_{t}(u):=i^{\ve_u}\Gamma_{\ve_u}(t-i\ell_u),
\quad\forall (t,u)\in \BR_{>0}\times \BR^\times
\end{equation}
and
\begin{equation}
  \label{ABx}
A:=\frac{x_{0,2}x_{1,3}}{x_{0,3}x_{1,2}},\quad B:=\frac{x_{0,3}x_{1,2}}{x_{0,1}x_{2,3}} \,.
\end{equation}

\subsection{Definition of the edge state-integral}
\label{sub.def}

In this section we define a Turaev--Viro type generalized TQFT based on the tetrahedral
weight~\eqref{Tweil3} for the self-dual LCA group $\BR^\times$. This follows
closely the definition of the of the KLV invariant and of the meromorphic 3D-index
(using an analogous function for the self-dual LCA groups $\BR \times \BR$ and
$\BT$, respectively~\cite{KLV,GK:mero}. 

Fix an ideal triangulation $\calT$ of an oriented 3-manifold with $N$ 
tetrahedra $T_i$ for $i=1,\dots, N$. The invariant is defined as follows:

\begin{itemize}
\item[(a)] 
Assign variables $x_i \in \BR^\times$ for $i=1,\dots,N$ to $N$ edges of $\calT$. 
\item[(b)] 
Choose a strictly positive pre-angle structure $\th=(a,b,c)$ 
at each tetrahedron, where $a+b+c=(1,-1)$ and $\dot a, \dot b, \dot c>0$.
Here, $a$ is the angle of the $01$ and $23$ edges, $b$ is the angle of the $02$
and $13$ edges, and $c$ is the angle of the $03$ and $12$ edges. 
\item[(c)]
  The weight $B(T,x,\th):=W_{a,c}(T,x)$ of a tetrahedron $T$ is given by
  Equation~\eqref{Tweil3}. 
\item[(d)]
Define
\be
\label{eq.Ipre}
\IE_{\hat\BR,\calT,\th} := \int_{(\BR^\times)^N} \prod_{i=1}^N B(T_i,x,\th)
\delta(x_N) d\mu(x)
\ee
where $d\mu(x)$ is the normalized Haar measure on $(\BR^\times)^N$.
\end{itemize}

An oriented unordered ideal tetrahedron has shape $z$, $z'=1/(1-z)$ and $z''=1-1/z$
placed at the pair of opposite edges $01$ and $23$, $02$ and $13$, or $03$ and $12$,
respectively. Recall the Neumann-Zagier matrices $\Abar$, $\Bbar$ and $\Cbar$
of an oriented, unordered ideal triangulation~\cite{NZ}, whose rows and columns are
indexed by the edges and the tetrahedra of the triangulation, respectively.
The $(i,j)$ entry of each of these matrices is the number of times the shape $z_j$
(or, $z_j'$, or $z_j''$) appears around the edge $e_i$ of the triangulation.

The next proposition shows that that the integral~\eqref{eq.Ipre} 
(and even the integrand) depends only on the Neumann--Zagier matrices of the
gluing equations of the triangulation $\calT$. We use the shorthand notation
$M_j$ for the $j$-th column of a matrix $M$ and $x^v=\prod_{i=1}^N x_i^{v_i}$. 

\begin{proposition}
\label{prop.ABC}
With the above notation and for $i=j,\dots,N$, we have:
\be
\label{eq.BTABC}
B(T_j,x,\th) = \sqrt{2\pi}
\frac{\operatorname{G}_{\dot a_j}\!\big(x^{(\Bbar-\Cbar)_j}\big)
\operatorname{G}_{\dot b_j}\!\big(x^{(\Cbar-\Abar)_j}\big)
\operatorname{G}_{\dot c_j}\!\big(x^{(\Abar-\Bbar)_j}\big)
}{\langle -\ddot a_j, x^{\Abar_j} \rangle
\langle -\ddot b_j, x^{\Bbar_j} \rangle
\langle -\ddot c_j, x^{\Cbar_j} \rangle}
\ee
It follows that $\IE_{\hat\BR,\calT,\th}$ depends on only the matrices $\Abar$, $\Bbar$,
$\Cbar$ and $\th$. Moreover,
\be
\label{abcprod}
\prod_{j-1}^N \langle -\ddot a_j, x^{\Abar_j} \rangle
\langle -\ddot b_j, x^{\Bbar_j} \rangle
\langle -\ddot c_j, x^{\Cbar_j} \rangle =1 
\ee
for all $(x,\th)$. Hence, this factor can be removed from the integrand
of~\eqref{eq.Ipre}. 
\end{proposition}

Its proof follows mutandis-mutandis from~\cite[Prop.3.1]{GK:mero}. To show
identity~\eqref{abcprod}, we use~\eqref{multixy} which implies that
$\langle -\ddot a_j x^{\Abar_j} \rangle = \prod_i \langle \ddot a_j^{\Abar_{ij}},x_i
\rangle^{-1}$, as well as the fact that $\th$ satisfies the edge-balancing
conditions $\prod_i \ddot a_j^{\Abar_{ij}} \ddot b_j^{\Bbar_{ij}} \ddot c_j^{\Cbar_{ij}} =1$
for all $i=1,\dots, N$. 

Following the arguments of~\cite[Sec.3]{GK:mero} \emph{mutatis mutandis}, we obtain 
an invariant $\IE_{\hat\BR,M}(\lambda,\mu)$ for a 3-manifold $M$ with
torus boundary.
  
\subsection{Examples}

We now illustrate the above invariant for the case of the complement of the
$4_1$ knot. The next lemma expresses the invariant of the $4_1$ knot in terms of
an integral of the beta function.

\begin{lemma}
\label{lem.41g2}
The invariant of the $4_1$ is given by
\begin{equation}
\label{41g2}
\IE_{\hat\BR,4_1}(\lambda,\mu)= \frac{1}{2\pi i}
\int_{\epsilon-i\BR}
\frac{\operatorname{B}(z,z-\dot\lambda)
\operatorname{B}(z+\dot\mu,z+\dot\mu-\dot\lambda)}{
\cos(\pi(z-\tfrac{\dot\lambda}2))\cos(\pi(z+\dot\mu-\tfrac{\dot\lambda}2))}
\big(\cos(\pi\tfrac{\dot\lambda}2)\big)^2
\operatorname{d}\!z
\end{equation}
where $\epsilon>0$. 
\end{lemma}

\begin{proof}
Using the  ideal triangulation of the $4_1$ knot with two tetrahedra, with
the balancing condition $2a_0+c_0=2a_1+c_1$ and the  definitions
$\lambda:=2a_0+c_0-\varpi$ and $\mu:=a_0-a_1$, we have

\begin{align*}
\IE_{\hat\BR,4_1}(\lambda,\mu)&=\int_{\BR^\times}
h_{a_0,c_0}(x, 1/x^2)\bar h_{a_1,c_1}(1/x, x^2)\frac{\operatorname{d}\!x}{|x|}\\
&=\frac12\int_{\BR^\times}
\frac{\langle x,\ddot c_0\ddot a_0^2\rangle}{\langle x,\ddot c_1\ddot a_1^2\rangle}
\frac{\Gamma_{\ve_{x}}(\dot a_0- i \ell_x)
\Gamma_{\ve_{x}}(\dot b_0-i \ell_{x})}{\Gamma_{0}(1-\dot c_0-2 i \ell_x)}
\frac{\Gamma_{\ve_{x}}(\dot a_1- i \ell_x)
\Gamma_{\ve_{x}}(\dot b_1-i \ell_{x})}{\Gamma_{0}(1-\dot c_1-2 i \ell_x)}
\operatorname{d}\!\ell_x\\
&=\frac1\pi\sum_{\epsilon\in\{0,1\}}\int_{\BR}
\prod_{k\in\{0,1\}}\operatorname{B}(\dot a_k- i t,\dot b_k- i t)
\frac{\cos(\tfrac{\pi}{2}(\epsilon+it-\dot a_k))
\cos(\tfrac{\pi}{2}(\epsilon+it-\dot b_k))}{\sin(\tfrac\pi2(\dot c_k+2 i t))}
\operatorname{d}\! t\\
&=\int_{\dot a_1-i\BR}
\frac{\operatorname{B}(z,z-\dot \lambda)
\operatorname{B}(z+\dot\mu,z+\dot\mu-\dot\lambda)}{
\pi\cos(\pi(z-\tfrac{\dot \lambda}2))\cos(\pi(z+\dot\mu-\tfrac{\dot \lambda}2))}\\
& \times \sum_{\epsilon\in\{0,1\}}
\cos(\tfrac{\pi}{2}(\epsilon-z))\cos(\tfrac{\pi}{2}(\epsilon+\dot \lambda -z))
\cos(\tfrac{\pi}{2}(\epsilon-\dot\mu-z))
\cos(\tfrac{\pi}{2}(\epsilon+\dot \lambda-\dot\mu -z))
\operatorname{d}\! iz\\
&=\frac{i}{2\pi}\int_{\dot a_1-i\BR}
\operatorname{B}(z,z-\dot\lambda)
\operatorname{B}(z+\dot\mu,z+\dot\mu-\dot\lambda)
\left(1+\frac{\big(\cos(\pi\frac{\dot\lambda}2)\big)^2}{
\cos(\pi(z-\frac{\dot\lambda}2))\cos(\pi(z+\dot\mu-\frac{\dot\lambda}2))}\right)
\operatorname{d}\! z \,.
\end{align*}
It remains to remove the $1$ above. The integral
$\int_{\dot a_1-i\BR} \operatorname{B}(z,z-\dot\lambda)
\operatorname{B}(z+\dot\mu,z+\dot\mu-\dot\lambda) \operatorname{d}\! z$ is \emph{not}
absolutely convergent since the integrand is $O(1/|z|)$ for large $z$. Yet,
the integral is distributionally well-defined and in fact vanishes; see
Lemma~\ref{lem.41per} below. This concludes the proof of the lemma.
\end{proof}

The next lemma uses the Fourier transform relation between the $g$ and the
$h$-functions~\eqref{eq:hac-gac} and expresses the invariant of the $4_1$ knot in
terms of a period of its $A$-polynomial.

\begin{lemma}
\label{lem.41g3}
We have: 
\be
\label{41g3}
\IE_{\hat\BR,4_1}(\lambda,\mu)
=
\int_{-\infty}^1
\frac{|1-x|^{\dot \mu }(|y+x-1|^{\dot \lambda}+|-y+x-1|^{\dot \lambda})}{
|x|^{2\dot \mu +\dot\lambda} 2^{\dot\lambda}} \frac{dx}{y}
\ee
where $y=\sqrt{(1-x)(1-x+4x^2)}$. 
\end{lemma}

\begin{proof}
Using the triangulation of $4_1$ with two tetrahedra, and its balancing edge
conditions, for all local fields $F$ we have: 
\begin{equation}
\IE_{\hat\BR,4_1}(\lambda,\mu)=\int_{\hat\sB}
h_{a_0,c_0}(\alpha, -2\alpha)\bar h_{a_1,c_1}(-\alpha, 2\alpha)\operatorname{d}\!\alpha.
\end{equation}
We now use the Fourier transform relation~\eqref{eq:hac-gac} and
obtain that
\begin{align*}
 \IE_{\hat\BR,4_1}(\lambda,\mu)
 &=\int_{\sB^4\times\hat\sB} g_{a_0,c_0}(x,y)(-\alpha)(x^2yu^2v)\bar g_{a_1,c_1}(u, v)
 \operatorname{d}(x,y,u,v)\operatorname{d}\!\alpha\\
 &=\int_{\sB^4} \delta_F(x^2yu^2v-1)g_{a_0,c_0}(x,y)\bar g_{a_1,c_1}(u, v)
 \operatorname{d}(x,y,u,v).
\end{align*}
Using the balancing condition $2a_0+c_0=2a_1+c_1$ and the explicit form of
$g_{a,c}(x,z)$ given in~\eqref{eq.weil1}, we can reduce the number of integrations
from four to two:
\begin{align*}
\IE_{\hat\BR,4_1}(\lambda,\mu) &=\int_{\sB^4} \delta_F(x^2yu^2v-1)
f_{\dot a_0,\dot c_0}(1/x,y)\bar f_{\dot a_1,\dot c_1}(1/u, v)
\operatorname{d}(x,y,u,v)\\
&=\int_{\sB^2} \delta_F\big(1-(1-x)(1-u)/(xu)^2\big)
\|x\|^{\dot c_0}\|1-x\|^{\dot a_0-1}\|u\|^{\dot c_1}\|1-u\|^{\dot a_1-1}
\operatorname{d}(x,u)\\
&=\int_{F^2} \delta_F\big((xu)^2-(1-x)(1-u)\big)
\|x\|^{\dot c_0+1}\|1-x\|^{\dot a_0-1}\|u\|^{\dot c_1+1}\|1-u\|^{\dot a_1-1}
\operatorname{d}(x,u) \,.
\end{align*}
Using the delta-function, we can replace the norm of $1-u$ in terms of the norms
of $x$, $1-x$ and $u$, and then use the definitions of the longitude $\lambda=a_0-b_0$
and the meridian $\mu=a_0-a_1$, and the change of variables $v=xu$ and $y=2xv-x+1$
to obtain that
\be
\label{41calc}
\begin{aligned}
\IE_{\hat\BR,4_1}(\lambda,\mu)
&=\int_{F^2} \delta_F\big((xu)^2-(1-x)(1-u)\big)
\big\|\tfrac{1-x}{x^2}\big\|^{\dot \mu }\|xu\|^{\dot \lambda}\operatorname{d}(x,u)\\
&=\int_{ F^2} \delta_F\big(v^2-(1-x)(1-\tfrac vx)\big)
\frac{\|1-x\|^{\dot \mu }\|v\|^{\dot \lambda}}{\|x\|^{2\dot \mu +1}}
\operatorname{d}\!x\operatorname{d}\!v\\
&=\int_{F^2} \delta_F\Big(y^2-(1-x)(1-x+4x^2)\Big)
\frac{\|1-x\|^{\dot \mu }\|y+x-1\|^{\dot \lambda}}{
\|x\|^{2\dot \mu +\dot\lambda}\|2\|^{\dot\lambda -1}}
\operatorname{d}\!x\operatorname{d}\!y \,.
\end{aligned}
\ee

When $F=\BR$, $1-x+4x^2>0$ for all real $x$, thus the delta function imposes
the condition that $x <1$. Part (a) of Lemma~\ref{lem.res0} for
\be
\label{p41}
p_{4_1}(x,y)=y^2-(1-x)(1-x+4x^2)
\ee
thought of as a function of $y$ with roots $y=\pm \sqrt{(1-x)(1-x+4x^2)}$ and
derivative $\frac{\partial}{\partial y} p_{4_1}(x,y) = 2y$, implies that
\begin{align*}
\IE_{\hat\BR,4_1}(\lambda,\mu)
&= \int_{-\infty}^1
\frac{|1-x|^{\dot \mu }(|y+x-1|^{\dot \lambda}+|-y+x-1|^{\dot \lambda})}{
|x|^{2\dot \mu +\dot\lambda} 2^{\dot\lambda}} \frac{dx}{y} \,,
\end{align*}
where $y=\sqrt{(1-x)(1-x+4x^2)}$. 
Over $\BC$, Equation~\eqref{p41} defines an elliptic
curve with Weiestrass form
\be
\label{A41weier}
Y^2 = -X^3 + \frac{1}{3} X + \frac{322}{27} 
\ee
and $j$-invariant $-\frac{1}{15}$. This elliptic curve is isomorphic to the
$A$-polynomial curve~\eqref{A41} of the $4_1$ knot.
\end{proof}

\begin{remark}
  \label{rem.41period}
  We can give an alternative period formula
\be
\label{41g3alt}
\IE_{\hat\BR,4_1}(\lambda,\mu)=
\int_{\BR} \frac{\big|(1-\tfrac74 t^2+t^4)^{1/2}-t\big|^{\dot \lambda}}{
|t-t^{-1}|^{2\dot\mu+\dot\lambda}(1-\tfrac74 t^2+t^4)^{1/2}}\operatorname{d}\!t \,.
\ee
for $\IE_{\hat\BR,4_1}(\lambda,\mu)$. To obtain this, use Equation~\eqref{41calc}
and do the change of variables $y=t \sqrt{1-x+4x^2}$ to obtain
\begin{align*}
\IE_{\hat\BR,4_1}(\lambda,\mu) &=
\int_{\BR^2} \delta_F\Big(y^2-(1-x)(1-x+4x^2)\Big)
\frac{|1-x|^{\dot \mu }|y+x-1|^{\dot \lambda}}{|x|^{2\dot \mu +\dot\lambda}2^{\dot\lambda -1}}
\operatorname{d}\!x\operatorname{d}\!y
\\
&=\int_{\BR^2} \delta_F\Big(t^2-1+x\Big)\frac{|1-x|^{\dot \mu }
\big|t\sqrt{1-x+4x^2}+x-1\big|^{\dot \lambda}}{2^{\dot\lambda -1}
|x|^{2\dot \mu +\dot\lambda}\sqrt{1-x+4x^2}}\operatorname{d}\!x \operatorname{d}\!t\\
&=\int_{\BR} \frac{\big|(1-\tfrac74 t^2+t^4)^{1/2}-t\big|^{\dot \lambda}}{
  |t-t^{-1}|^{2\dot\mu+\dot\lambda}(1-\tfrac74 t^2+t^4)^{1/2}}\operatorname{d}\!t
\end{align*}
This expresses the function $\IE_{\hat\BR,4_1}(\lambda,\mu)$ in terms of a period
of the elliptic curve $y^2=4 -7 t^2 + 4 t^4$. The latter has $j$-invariant
$13997521/225$, hence it is not isomorphic to the $A$-polynomial curve. 
\end{remark}

Combining the above lemmas~\ref{lem.41g2} and~\ref{lem.41g3} when
$\lambda=\mu=0$, we obtain Equation~\eqref{41identity} stated in the introduction.


\section{Fourier transforms of the Euler $\Gamma$ and
$\operatorname{B}$-functions}

In this section, which is independent of quantum dilogarithms and of local
fields, we give some complementary properties of the Fourier transform of the
$\Gamma$ and $\operatorname{B}$-functions which explain the relation of
multidimensional Mellin-Barnes integrals to periods of algebraic varieties.
Our results are similar to the work of Passare--Tsikh--Cheshel~\cite{Passare:mellin}
who relate periods of families of Calabi-Yau manifolds to Mellin--Barnes integrals
and explains why the explicit Mellin--Barnes integrals that appear in~\cite{HKS}
are periods of the $A$-polynomial curves.

We begin with an elementary lemma (which ought to be better-known) that
computes the inverse Fourier transform of the $\Gamma$ and the
$\operatorname{B}$-functions. 

\begin{lemma}
\label{BG}
For $a>0$ and $x$ real we have:
\be
\label{ig1}
\Gamma(a+2 \pi i x) = \int_\BR e^{2 \pi i x s} e^{a s -e^{ s}} \operatorname{d}\!s 
\ee
For $a,b >0$ and $x,y$ real we have:
\be
\label{ib1}
\begin{aligned}
  \operatorname{B}(a+2 \pi i x, a+2 \pi i y) &= \int_{\BR^3}
  \frac{e^{-2 \pi i (xu + yv)}}{
  (1+e^s)^{a} (1+e^{-s})^{b}} 
  \delta(u-\log(1+e^s))  \delta(v-\log(1+e^{-s})) \operatorname{d}\!s
  \operatorname{d}\!u \operatorname{d}\!v 
\end{aligned}
\ee
as well as
\be
\label{ib2}
\begin{aligned}
  \operatorname{B}(a+2 \pi i x, a+2 \pi i x) &= \int_{\BR^2}
  \frac{e^{-2 \pi i x t}}{
  (1+e^s)^{a} (1+e^{-s})^{a}} 
  \delta(t-\log(1+e^s)\log(1+e^{-s})) \operatorname{d}\!s
  \operatorname{d}\!t
\end{aligned}
\ee
\end{lemma}

\begin{proof}
The definition of the $\Gamma$-function and the change of variables $t=e^s$
gives:
\begin{equation*}
  \Gamma(z) = \int_0^\infty t^z e^{-t} \frac{\operatorname{d}\!t}{t}
  = \int_\BR e^{s z} e^{-e^s} ds \,.
\end{equation*}
Then, substitute $z=a+2\pi i x$ to obtain~\eqref{ig1}, and observe that the
integral is absolutely convergent (in fact, the integrand is exponentially decaying
for real $s$ with $|s|$ large.

For the next identity, use
\be
\operatorname{B}(z,w)=\int_0^1 (1-t)^z t^w \frac{\operatorname{d}\!t}{t(t-1)}
\ee
and observe that
$$
\frac{\operatorname{d}\!t}{t(t-1)} = \big(\frac{1}{t}+\frac{1}{1-t}\big)
\operatorname{d}\!t = \operatorname{d}(\log(t)-\log(1-t))=
\operatorname{d}\! \log\big(\frac{t}{1-t}\big)=\operatorname{d}\! s
$$
where $s=\log\big(\frac{t}{1-t}\big)$ satisfies $t=1/(1+e^{-s})$ and
$1-t=1/(1+e^{s})$. The change of variables from $t$ to $s$ gives
\begin{align*}
\operatorname{B}(z,w)&=\int_\BR (1+e^s)^{-a} (1+e^{-s})^{-b}
e^{-2 \pi i (x \log(1+e^s) + y \log(1+e^{-s}))} \operatorname{d}\!s
\\
&=\int_{\BR^3}
\frac{e^{-2 \pi i (xu + yv)}}{
  (1+e^s)^{a} (1+e^{-s})^{b}} 
  \delta(u-\log(1+e^s))  \delta(v-\log(1+e^{-s})) \operatorname{d}\!s
  \operatorname{d}\!u \operatorname{d}\!v
\end{align*}
which concludes the proof of Equation~\eqref{ib1}. The proof of~\eqref{ib2}
is similar and left to the reader.
\end{proof}

%

An application of the above lemma is the vanishing of the following
distributional integral. 

\begin{lemma}
\label{lem.41per}  
For all $a>0$, we have:
\be
\label{41per}
\int_{a-i \BR} \operatorname{B}(z,z)^2 \operatorname{d}\!z=0 \,.
\ee
\end{lemma}

Note that $|\operatorname{B}(z,z)^2|=O(|z|^{-1})$ for
$z=a-i x$ with $t$ real and $|x|$ large, hence the integral~\eqref{41per}
is not absolutely convergent.

\begin{proof}
Equation~\eqref{ib2} expresses $\operatorname{B}(z,z)$ as a double integral
distributionally
$$
  \operatorname{B}(a+2 \pi i x, a+2 \pi i x) = \int_{\BR}
  \Big(\frac{e^s}{(1+e^s)^2} \Big)^a e^{-2 \pi i x \log(1+e^s)\log(1+e^{-s})}
  \operatorname{d}\!s \,.
$$
Using this identity twice for $z=a-i x$, we obtain that the integral
$I$ of Equation~\eqref{41per} is given by
\begin{align*}
  I &= \int_{\BR^3} \Big(\frac{e^s}{(1+e^s)^2} \Big)^a
  e^{-2 \pi i x \log(1+e^s)\log(1+e^{-s})}  \Big(\frac{e^t}{(1+e^t)^2} \Big)^a
  e^{-2 \pi i x \log(1+e^t)\log(1+e^{-t})}
 \operatorname{d}\!s \operatorname{d}\!t \operatorname{d}\!x
\end{align*}
Now perform the $x$-integral, which is a single Fourier transform, to obtain
that
\begin{align*}
  I &= \int_{\BR^3} \Big(\frac{e^s}{(1+e^s)^2} \Big)^a
  \Big(\frac{e^t}{(1+e^t)^2} \Big)^a \delta(\log(1+e^s)\log(1+e^{-s})
  +\log(1+e^t)\log(1+e^{-t}))
  \operatorname{d}\!s \operatorname{d}\!t \,.
\end{align*}
Since $\log(1+e^s)\log(1+e^{-s})+\log(1+e^t)\log(1+e^{-t}) >0$ for all real $s$
and $t$, the distribution vanishes, and the result follows. 
\end{proof}

Consider the integral
\be
\label{52HKS}
\kappa^{\mathrm{HKS}}_{5_2} =
\frac{1}{(2\pi i)^2} \int_{i \BR^2}
\operatorname{B}^2\bigl(\tfrac{1}{2}-x,\tfrac{1}{3}+x-y\bigr)
\operatorname{B}\bigl(\tfrac{1}{2}-x,\tfrac{1}{3}+2y\bigr)
\operatorname{d}\!x \operatorname{d}\!y 
\ee
from~\cite[Fig.10.2]{HKS}, whose numerical value is $\kappa^{\mathrm{HKS}}_{5_2}
=.534186\dots$.

\begin{proposition}
\label{prop.52constant}
We have:
\be
\label{52a}
\begin{aligned}
\kappa^{\mathrm{HKS}}_{5_2} &= \int_{\BR^6}
e^{-\frac{1}{2} u_1 -\frac{1}{3} v_1 -\frac{1}{2} u_2 -\frac{1}{3} v_2
  -\frac{1}{2} u_3 -\frac{1}{3} v_3} \delta(u_1-v_1+u_2-v_2+u_3)
\delta(v_1+v_2-2v_3) \\
& \qquad \times \prod_{j=1}^3 \delta(e^{-u_j}+e^{-v_j}-1)
\operatorname{d}\!u_j \operatorname{d}\! v_j \,.
\end{aligned}
\ee
Letting $(z_j,w_j)=(e^{-u_j},e^{-v_j})$ for $j=1,2,3$ consider the curve
$X(\BC^\times)$ in $(\BC^\times)^6$ with defining equations
\be
\label{X52}
\begin{aligned}
  z_j + w_j &=1 & j=1,2,3 &, \\
  z_1 w_1^{-1} z_2 w_2^{-1} z_3 & =1, & w_1 w_2 w_3^{-2} &=1 \,.  
\end{aligned}
\ee
Then, $\kappa^{\mathrm{HKS}}_{5_2}$ is an integral over the cycle $X(\BR_+)$.
\end{proposition}

\begin{proof}
Using~\eqref{ib2}, we have:
\begin{align*}
  \operatorname{B}\bigl(\tfrac{1}{2}-ix,\tfrac{1}{3}+ix-iy\bigr) &=
  \int_{\BR^3} \frac{e^{i x u_1 -i(x-y)v_1}}{
    (1+e^{s_1})^{\frac{1}{2}} (1+e^{-s_1})^{\frac{1}{3}}} \\ & \qquad
  \delta(u_1 -\log(1+e^{s_1})) \delta(v_1 -\log(1+e^{-s_1}))
  \operatorname{d}\!s_1 \operatorname{d}\! u_1
  \operatorname{d}\! v_1
\\
\operatorname{B}\bigl(\tfrac{1}{2}-ix,\tfrac{1}{3}+2iy\bigr) &=
  \int_{\BR^3} \frac{e^{i x u_3 -2iy v_3}}{
    (1+e^{s_1})^{\frac{1}{2}} (1+e^{-s_1})^{\frac{1}{3}}} \\ & \qquad
  \delta(u_3 -\log(1+e^{s_3})) \delta(v_3 -\log(1+e^{-s_3}))
  \operatorname{d}\!s_3 \operatorname{d}\! u_3
  \operatorname{d}\! v_3
\end{align*}
Insert twice the first equation and once the second equation
to~\eqref{52HKS} to obtain a 11-dimensional integral representation of
$\kappa^{\mathrm{HKS}}_{5_2}$. Now, do the $x$ and $y$ integration (which is a
Fourier transform), which reduces the integral to a 9-dimensional one and inserts
the product of two delta functions
$\delta(u_1-v_1+u_2-v_2+u_3) \delta(v_1+v_2-2v_3)$ in the integrand.

Using part (a) of Lemma~\ref{lem.res0}, we see that
for every function of Schwartz-Bruhat class $g$, we have
$$
\int_\BR g(s) \delta(u -\log(1+e^{s})) \delta(v -\log(1+e^{-s}))
\operatorname{d}\!s = g(\log(e^u-1)) \delta(e^{-u}+e^{-v}-1) \,. 
$$
Applying the above identity concludes the proof of~\eqref{52a}.


We now study the solution to the delta function equations. 
In the complex torus $(\BC^\times)^6$, with coordinates
$(z_1,z_2,z_3,w_1,w_2,w_3)$, where $(z_j,w_j)=(e^{-u_j},e^{-v_j})$ for $j=1,2,3$,
the equations~\eqref{X52} define a curve $X(\BC^\times)$ given by the equation
\be
\label{p52}
-1 + 2 z_1 - z_1^2 + 2 z_2 - 2 z_1 z_2 - z_2^2 + z_1^2 z_2^2 - z_1^3 z_2^2 - 
z_1^2 z_2^3 + z_1^3 z_2^3 = 0 \,.
\ee
The above equation has discriminant with respect to $z_2$ a polynomial in
$z_1$ with real roots at $-6.44292\dots$, $0$ and $1$. 
Moreover, the points $X(\BR_+)$ in the curve with coordinates in $\BR_+^6$ 
are parametrized by $z_1 \in (0,1)$, $z_2=z_2(z_1)$ being the unique real
branch of~\eqref{p52} for all $z_1 \in \BR$, and
$z_3=\tfrac{1 - z_1 - z_2 + z_1 z_2}{z_1 z_2}$. Combined with~\eqref{52a},
this expresses $\kappa^{\mathrm{HKS}}_{5_2} = \int_{X((0,1))} \omega$
where $\omega$ is a holomorphic differential form on the curve~\eqref{p52}.
A final computation identifies this curve with the $A$-polynomial of the
$5_2$ knot. In particular, this gives a proof of~\cite[Eqn.(18)]{GW:asy3D}. 
\end{proof}

\subsection*{Acknowledgements} 
S.G. wishes to thank Frank Calegari, Nathan Dunfield and Francesco Campagna for
enlightening conversations, and the Max-Planck-Institute in Bonn and the University
of Geneva for their hospitality during which the paper was completed. 
R.K. wishes to thank Oliver Braunling for enlightening conversations. R.K. 
was supported in part by the Swiss National Science Foundation, grant no.
200020-200400, and the Russian Science Foundation, subsidy no~21-41-00018.


\bibliographystyle{plain}
\bibliography{biblio}
\end{document}